\documentclass[12pt]{amsart}

\usepackage[margin=3cm]{geometry}

\title{Average Cost Optimality of Partially Observed MDPs: Contraction of Non-Linear Filters, Optimal Solutions and Approximations}

\author{Yunus Emre Demirci}
\address{Queen's University, Department of Mathematics and Statistics, Kingston, Ontario, Canada}
 \email{21yed@queensu.ca}
 
\author{Ali Devran Kara}
\address{University of Michigan Ann Arbor, Department of Mathematics} 
 \email{alikara@umich.edu}  
 
 \author{Serdar Yüksel}
\address{Queen's University, Department of Mathematics and Statistics, Kingston, Ontario, Canada} 
 \email{yuksel@queensu.ca}  
 
\RequirePackage{amsthm,amsmath,amsfonts,amssymb}
\RequirePackage[numbers]{natbib}

\newcommand{\bea}{\begin{eqnarray}}
\newcommand{\ena}{\end{eqnarray}}
\newcommand{\beas}{\begin{eqnarray*}}
\newcommand{\enas}{\end{eqnarray*}}
\newcommand{\beq}{\begin{equation}}
\newcommand{\enq}{\end{equation}}

\newtheorem{theorem}{Theorem}[section]
\newtheorem{corollary}{Corollary}[section]

\newtheorem{lemma}{Lemma}[section]
\newtheorem{definition}{Definition}[section]
\newtheorem{example}{Example}[section]

\newtheorem{defn}{Definition}[section]

\newtheorem{note}{Note}[section]

\newcommand\norm[1]{\left\lVert#1\right\rVert}

\newcommand{\T}{\mathcal{T}}
\newcommand{\X}{\mathbb{X}}
\newcommand{\Y}{\mathbb{Y}}
\newcommand{\U}{\mathbb{U}}
\usepackage{graphics} 

\usepackage{epsfig} 

\usepackage{mathptmx} 

\usepackage{times} 

\usepackage{amsmath} 

\usepackage{amssymb}  

\usepackage{dsfont}

\usepackage{mathtools}

\usepackage[utf8]{inputenc}

\usepackage{graphicx}

\usepackage{newtxtext, newtxmath}

\usepackage{verbatim}

\usepackage{amsfonts}

\usepackage{enumitem}

\usepackage{bbm}

\newcommand{\R}{\mathbb{R}}

\newcommand{\1}{\mathbbm{1}}

\newcommand{\N}{\mathbb{N}}

\let\phi\varphi

\newcommand{\Z}{\mathcal{Z}}

\newcommand{\B}{\mathbb{B}}

\newcommand{\sT}{\mathcal{T}}

\newcommand{\sZ}{\mathcal{Z}}

\newcommand{\Zplus}{\mathbb{Z}_{\geq 0}}

\newcommand{\F}{\mathbb{F}}

\newcommand{\C}{\mathbb{C}}

\newcommand{\Pf}{\mathcal{P}} 

\newtheorem{assumption}{Assumption}

\usepackage[colorlinks=true,breaklinks=true,bookmarks=true,urlcolor=blue,
citecolor=blue,linkcolor=blue,bookmarksopen=false,draft=false]{hyperref}

\newcommand{\V}{\mathcal{V}}
\let\phi\varphi
\newcounter{tempcounter}

\usepackage{lipsum}
\usepackage{amsfonts}
\usepackage{graphicx}
\usepackage{epstopdf}
\usepackage{algorithmic}
\ifpdf
  \DeclareGraphicsExtensions{.eps,.pdf,.png,.jpg}
\else
  \DeclareGraphicsExtensions{.eps}
\fi


\begin{document}



\begin{abstract}   
  The average cost optimality is 
  known to be a challenging problem 
  for partially observable stochastic control, 
  with few results available beyond 
  the finite state, action, and measurement setup, for which somewhat restrictive conditions are available. 
  In this paper, we present explicit 
  and easily testable conditions 
  for the existence of solutions to 
  the average cost optimality equation 
  where the state space is compact. 
  In particular, we present a new contraction based 
  analysis, which is new to the literature 
  to our knowledge, building on recent 
  regularity results for non-linear filters.
  Beyond establishing existence, we also present 
  several implications of our analysis that are new to the literature: (i) robustness to incorrect priors, (ii) near optimality of policies based on 
  quantized approximations, (iii) near optimality of policies with finite memory, and (iv) convergence in Q-learning. In addition to our main theorem, each of these represents a novel contribution for average cost criteria.  
\end{abstract}

\maketitle

\section{Introduction}

We study optimal control for partially observable 
Markov decision processes (PODMPs) under the 
average cost criterion. 
Let $\X$ denote a standard Borel space, 
which is the state space of a partially 
observed controlled Markov process. 
Let $\mathbb{B}(\mathbb{X})$ be its Borel 
$\sigma$-field.
Let $\mathbb{C}_b(\mathbb{X})$ be the set of all 
continuous, bounded functions on $\mathbb{X}$. 
Here and throughout the 
paper $\mathbb{Z}_+$ denotes the set of non-negative
integers and $\N$ 
denotes the set of positive integers. Let
$\Y$ be a standard Borel space denoting the 
observation space of the model, 
and let the state be observed through an
observation channel $Q$. 
The observation channel, $Q$, is 
defined as a stochastic kernel (regular
conditional probability) from  
$\X$ to $\Y$, such that
$Q(\,\cdot\,|x)$ is a 
probability measure on the power set 
$P(\Y)$ of $\Y$ for every $x
\in \mathbb{X}$, and $Q(A|\,\cdot\,): \X \to [0,1]$ 
is a Borel
measurable function for every $A \in P(\Y)$.  A
decision maker (DM) is located at the output 
of the channel $Q$, and hence only sees 
the observations $\{Y_t,\, t\in \Zplus\}$ and choosing its actions 
from $\U$, the action space which is a compact set. 
An {\em admissible policy} $\gamma$ is a
sequence of control functions $\{\gamma_t,\, t\in \Zplus\}$ such
that $\gamma_t$ is measurable with respect to the $\sigma$-algebra
generated by the information variables
$
I_t=\{Y_{[0,t]},U_{[0,t-1]}\}, \quad t \in \mathbb{N}, \quad
I_0=\{Y_0\},
$
where
\begin{equation}
\label{eq_control}
U_t=\gamma_t(I_t),\quad t\in \Zplus,
\end{equation}
are the $\mathbb{U}$-valued control
actions and 
$Y_{[0,t]} = \{Y_s,\, 0 \leq s \leq t \}, \quad U_{[0,t-1]} =
\{U_s, \, 0 \leq s \leq t-1 \}.$
In the above, the dependence of control policies on the 
initial distribution $\pi_0$ is implicit. We will 
denote the collection of admissible control policies as $\Gamma$.
The update rules of the system are 
determined by (\ref{eq_control}) and the following
relationships:
\[  \Pr\bigl( (X_0,Y_0)\in B \bigr) =  \int_B \mu(dx_0)Q(dy_0|x_0), \quad B\in \mathcal{B}(\mathbb{X}\times\mathbb{Y}), \]
where $\mu$ is the (prior) distribution of the initial state $X_0$, and
\begin{eqnarray*}
\label{eq_evol}
 &\Pr\biggl( (X_t,Y_t)\in B \, \bigg|\, (X,Y,U)_{[0,t-1]}=(x,y,u)_{[0,t-1]} \biggr) \\
& = \int_B Q(dy_t|x_t)   \mathcal{T}(dx_t|x_{t-1}, u_{t-1}),  
\end{eqnarray*}
$B\in \mathcal{B}(\mathbb{X}\times\mathbb{Y}), t\in \mathbb{N},$ 
where $\sT$ is the transition 
kernel of the model which is a stochastic 
kernel from $\X\times\U$ to $\mathbb{X}$. 

We may let the agent's goal be to minimize the expected discounted cost
$$
J_\beta(\mu, \gamma)=E_\mu^{\gamma}\left[\sum_{t=0}^{\infty} \beta^t c\left(X_t, U_t\right)\right]
$$
for some discount factor $\beta \in(0,1)$, over the set of admissible policies $\gamma \in \Gamma$, where $c: \mathbb{X} \times \mathbb{U} \rightarrow \mathbb{R}$
is the stage-wise measurable cost function, 
and the expectation $E_\mu^{\gamma}$ is taken over the initial state probability measure $\mu$ under policy $\gamma$.
The optimal cost for the discounted infinite horizon is defined as 
$$J_\beta^*(\mu)=\inf _{\gamma \in \Gamma} J_\beta(\mu, \gamma).$$

The average cost control problem 
under partial observations involves 
finding an optimal policy that minimizes 
the average cost of the system over an infinite horizon:
$$J^*(\mu)=\inf _{\gamma \in \Gamma} J(\mu, \gamma)$$
where
$$
J(\mu, \gamma)=\limsup_{n \to \infty}\frac{1}{n}E_\mu^{\gamma}\left[\sum_{t=0}^{n-1} c\left(X_t, U_t\right)\right]
.$$

{\bf Contribution:}

\begin{itemize}
\item[(i)] In the theory of Partially Observable Markov 
Decision Processes (POMDPs), conditions for the existence of solutions to 
the average cost optimality equation (ACOE) have few findings in the literature,
especially in scenarios beyond those with finite states, 
actions, and measurements, often restricted by stringent 
conditions, as we discuss in detail in the paper.
This paper introduces a novel result concerning 
the existence of a solution for the ACOE,
 by utilizing a Wasserstein 
regularity result (Theorem \ref{ergodicity}), 
as detailed in Theorem \ref{main}.
Under Assumption \ref{main_assumption},
we show the existence of a 
solution for the ACOE. We provide a detailed comparison with the existing results in the literature and highlight the explicit nature of our conditions.

\item[(ii)] Subsequently, the paper then presents several applications and implications of the existence of a solution to ACOE for POMDPs. These lead to new contributions in this area of study to the best of our knowledge:
\begin{enumerate}
\item Subsection \ref{robustness} establishes robustness to incorrect priors under universal filter stability, demonstrating 
that an optimal policy designed for an incorrect prior remains optimal when applied for the correct prior
under the average cost criteria, leading to a complete robustness property. Additionally, the subsection includes an analysis 
for the discounted cost setup.
\item Subsection \ref{Near_opt_quant} establishes near-optimality of quantized approximation 
policies for average cost criteria.
\item Subsection \ref{Near_opt_fW} examines how, 
under conditions of filter stability, optimal policies 
derived from a finite window of measurements and actions for the 
discounted cost criteria are near-optimal solutions 
for average cost criteria in POMDPs.
\item We also highlight the use of Q-learning to 
derive near-optimal policies applicable for 
both discounted finite window and discounted 
quantized approximation scenarios.

\end{enumerate}
\end{itemize}

\subsection{Notation and preliminaries}

\hfill\break

{\bf Belief MDP reduction for POMDPs.}
It is known that any POMDP can be reduced to a completely observable Markov process \cite{Yus76}, \cite{Rhe74}, whose states are the posterior state distributions or {\it belief}s of the observer; that is, the state at time $n$ is
\begin{align}
z_n(\,\cdot\,) := P\{X_{n} \in \,\cdot\, | y_0,\ldots,y_n, u_0, \ldots, u_{n-1}\} \in {\mathcal P}(\mathbb{X}). \nonumber
\end{align}
We call this equivalent process the filter process. 
We denote by
$\Z:={\mathcal P}(\mathbb{X})$ 
the set of probability measures on 
$(\mathbb{X}, \mathbb{B}(\mathbb{X}))$ 
under the weak convergence topology, where, 
under this topology $\mathcal{Z}$ 
is also a standard Borel space, that is, $\Z=\mathcal{P}(\mathbb{X})$ is separable and completely metrizable under the weak convergence topology.  
The filter process has state space 
$\mathcal{Z}$ and 
action space $\mathbb{U}$.
Let $\Pf(\mathcal{Z})$ denote the 
probability measures on $\mathcal{Z}$, 
equipped with the weak convergence topology. 

The transition probability $\eta$ of the filter 
process can be determined via the 
following equation 
\cite{HernandezLermaMCP, Rhe74, Yus76}: 
\begin{align}\label{conteta}
\eta(\cdot \mid z, u)=\int_{\mathbb{Y}} 1_{\{F(z, u, y) \in \cdot\}} P(d y \mid z, u),
\end{align}
where 
$$P(\cdot \mid z, u)=\operatorname{Pr}\left\{Y_{n+1} \in \cdot \mid Z_n=z, U_n=u\right\}$$ 
from $\mathcal{Z} \times \mathbb{U}$ to $\mathbb{Y}$ and 
$$
F(z, u, y):=\operatorname{Pr}\left\{X_{n+1} \in \cdot \mid Z_n=z, U_n=u, Y_{n+1}=y\right\}
$$
from $\mathcal{Z} \times \mathbb{U} \times \mathbb{Y}$ to $\mathcal{Z}$.


The one-stage cost function $\tilde{c}:\mathcal{Z} \times \mathbb{U} \rightarrow \mathbb{R}$ is a Borel measurable function and it is given by
\begin{align}\label{tildecost}
\tilde{c}(z, u):=\int_{\mathbb{X}} c(x, u) z(d x),
\end{align}
where $c: \mathbb{X} \times \mathbb{U} \rightarrow \mathbb{R}$ is the stage-wise cost function.

This way, we obtain a completely observable Markov decision process from the POMDP, with the components 
$(\mathcal{Z}, \mathbb{U}, \tilde{c}, \eta)$. The resulting MDP is often referred to as the belief-MDP.

{\bf Convergence notions for probability measures.}
Let $\{\mu_n,\, n\in \mathbb{N}\}$ be a sequence in
$\mathcal{P}(\mathbb{X})$.
The sequence $\{\mu_n\}$ is said to  converge
to $\mu\in \mathcal{P}(\mathbb{X})$ \emph{weakly} if
\begin{align}\label{weakConvD}
 \int_{\mathbb{X}} f(x) \mu_n(dx)  \to \int_{\mathbb{X}}f(x) \mu(dx)
\end{align}
for every continuous and bounded $f: \mathbb{X} \to \mathbb{R}$.

For two probability measures $\mu,\nu \in
\mathcal{P}(\mathbb{X})$, the \emph{total variation} metric
is given by
\begin{align}
\|\mu-\nu\|_{TV}:= & 2 \sup_{B \in {\mathcal B}(\mathbb{X})}
|\mu(B)-\nu(B)| \nonumber \\
 =&  \sup_{f: \, \|f\|_{\infty} \leq 1} \bigg| \int f(x)\mu(dx) -
\int f(x)\nu(dx) \bigg|, \label{TValternative}
\end{align}
where the supremum is over all measurable real $f$ such that
$\|f\|_{\infty} = \sup_{x \in \mathbb{X}} |f(x)|\le 1$.

Finally, the bounded-Lipschitz metric 
$\rho_{BL}$ \cite[p.109]{villani2008optimal} 
can also be used to metrize weak convergence:
\begin{align}
\rho_{BL}(\mu,\nu) = \sup_{\|f\|_{BL}\leq1} \biggl| \int_{\mathbb{X}} f(x) \mu(dx) - \int_{\mathbb{X}} f(x) \nu(dx) \biggr|, \label{BLD}
\end{align}
where
\begin{align}
\|f\|_{BL} := \|f\|_{\infty} + \norm{f}_{L}, 
\quad \norm{f}_{L}= \sup_{x \neq x'} \frac{f(x) - f(x')}{d_{\mathbb{X}}(x,x')} \nonumber
\end{align}
and $d$ is the metric on $\mathbb{X}$. 

When $\X$ is compact, one way to metrize $\mathcal{Z}$ under 
the weak convergence topology is via
the Kantorovich-Rubinstein metric 
(also known as the Wasserstein metric of order $1$) 
(\cite{Bog07}, Theorem 8.3.2) 
defined as follows 
\begin{align}\label{defkappanorm}
&W_1(\mu, \nu):=\sup \left\{\int_{\mathbb{X}} f(x) \mu(d x)-\int_{\mathbb{X}} f(x) \nu(d x) : f \in \operatorname{Lip}(\mathbb{X},1) \right\},
\end{align}
$\mu, \nu \in \mathcal{Z}$, where for $k \in \mathbb{N}$,
\begin{align*}
\operatorname{Lip}(\mathbb{X},k)=\{f:\mathbb{X}\to \mathbb{R},\; \norm{f}_{L}\leq k\}.
\end{align*}
\begin{definition}\label{Dobrushincoefficient}
    [\cite{dobrushin1956central}, Equation 1.16.] 
    For a kernel operator $K: S_1 \rightarrow \mathcal{P}\left(S_2\right)$ 
we define the Dobrushin coefficient as:
$$
\delta(K)=\inf \sum_{i=1}^n \min \left(K\left(x, A_i\right), K\left(y, A_i\right)\right)
$$
where the infimum is over all $x, y \in S_1$ and all partitions $\left\{A_i\right\}_{i=1}^n$ of $S_2$.
\end{definition}

\subsection{Average Cost Optimality Equation}\label{acoe_sec} 
Average cost optimality equation (ACOE) plays a crucial role for the analysis and the existence results of MDPs under the infinite horizon average cost optimality criteria. In the framework of the belief-MDP noted above, the triplet $(h,\rho^*,\gamma^*)$ where 
$h:\mathcal{Z}\to \R$,
$\gamma^*:\Z \to \U$
are 
measurable functions and $\rho^*\in\R$ is a 
constant, forms the ACOE if 
\begin{align}\label{acoe}
h(z)+\rho^* &=\inf_{u\in\U}\left\{\tilde{c}(z,u) + \int h(z_1)\eta(dz_1|z,u)\right\}\nonumber\\
&=\tilde{c}(z,\gamma^*(z)) + \int h(z_1)\eta(dz_1|z,\gamma^*(z))
\end{align}
for all $z\in\mathcal{Z}$. It is well known that (see e.g. \cite[Theorem 5.2.4]{HernandezLermaMCP}) if (\ref{acoe}) is satisfied with the triplet $(h,\rho^*,\gamma^*)$, and furthermore if $h$ satisfies
\begin{align}\label{h_inf}
\sup_{\gamma\in\Gamma}\lim_{t \to \infty}\frac{E_z^\gamma[h(Z_t)] }{t}=0, \quad \forall z\in\mathcal{Z}
\end{align}
then $\gamma^*$ is an optimal policy for the POMDP under the infinite horizon average cost optimality criterion, and 
\begin{align*}
J^*(z)=\inf_{\gamma\in\Gamma}J(z,\gamma)=\rho^* \quad \forall z\in \mathcal{Z}.
\end{align*}
We will refer to the function $h$, the relative value function, for the rest of the paper. Note that there may not be a unique relative value function $h$ that satisfies the ACOE, however, any $h$ that satisfies the ACOE and the condition (\ref{h_inf}) can be used for optimality analysis.

\subsection{Statement of the Main Result}
Now we state the main result of our paper.
\begin{assumption}\label{main_assumption}
\noindent
\begin{enumerate}
\item \label{compactness}
$U$ is a compact space and $(\X, d)$ is a compact metric space 
with diameter $D$ (where $D=\sup_{x,y \in \mathbb{X}} d(x,y)$).
\item \label{totalvar}
The transition probability $\sT(\cdot \mid x, u)$ is 
continuous in total variation in $(x, u)$, i.e., 
for any $\left(x_n, u_n\right) \rightarrow(x, u), 
\sT\left(\cdot \mid x_n, u_n\right) \rightarrow 
\sT(\cdot \mid x, u)$ in total variation.
\item \label{regularity}
There exists 
$\alpha \in R^{+}$such that 
$$
\left\|\mathcal{T}(\cdot \mid x, u)-\mathcal{T}\left(\cdot \mid x^{\prime}, u\right)\right\|_{T V} \leq \alpha d(x, x^{\prime})
$$
for every $x,x' \in \mathbb{X}$, $u \in \mathbb{U}$.
\item \label{CostLipschitz}
There exists $K_1 \in \mathbb{R}^+$ such that
\[|c(x,u) - c(x',u)| \leq K_1 d(x,x').\]
for every $x,x' \in \mathbb{X}$, $u \in \mathbb{U}$.
\item The cost function $c$ is continuous, and thus bounded since $\mathbb{X}$ is assumed to be compact.
\item \label{dobr_coef}
$$K_2:=\frac{\alpha D (3-2\delta(Q))}{2} < 1,$$
where $\delta(Q)$ is defined as in Definition \ref{Dobrushincoefficient}.
\end{enumerate}
\end{assumption}
We thus state the following main theorem.
\begin{theorem}\label{main} 
    Under Assumption \ref{main_assumption}, 
    a solution to the average cost optimality 
    equation (ACOE) exists. 
    This leads to the existence of an optimal 
    control policy, and optimal cost is constant for 
    every initial state.
\end{theorem}

As can be seen, the testability/verification of 
these criteria is explicit. We provide three examples to illustrate this, 
one in a discrete setting and the others
in continuous setting.

As we will discuss in more detail in Section 
\ref{impl}, 
the theorem provides significant implications 
for robustness and approximations. By utilizing this 
theorem, we can derive near-optimal policies 
for the average cost criteria in partially 
observable Markov decision processes. 
To the best of our knowledge, obtaining 
near-optimal policies for the average cost 
criteria is a novel result in the literature
on Partially Observable Markov Decision Processes (POMDPs). 
Additionally, under certain conditions, our 
findings contribute a new perspective to the 
literature on robustness, demonstrating that 
the average cost optimization problem is 
completely robust to initialization errors.

\subsection{Examples}
The first example will be for the discrete case.
For the case with finite $\X$, consider 
the discrete metric $d$ 
defined as follows:
$$d\left(x, x^{\prime}\right)= \begin{cases}1 & \text { if } x \neq x^{\prime} \\ 0 & \text { if } x=x^{\prime}.\end{cases}$$
With this choice of metric, the diameter $D$ is equal to $1$.
\begin{example}\label{ex1}
        Let $\X=\{0, 1, 2, 3\}$, $\Y=\{0, 1\}$,
        $\U=\{0, 1\}$, $\epsilon \in (0,1/2)$, and
        let $c$ be any function from $X \times U$ to $R^+$. 
        Now, consider the transition and measurement matrices given by:
        \begin{align*}
        \sT_0=\left(\begin{array}{cccc}
        1/2 & 1/3 & 1/6 & 0 \\
        0 & 1/2 & 1/6 & 1/3 \\
        1/2 & 1/6 & 0 & 1/3 \\
        1/3 & 1/3 & 1/3 & 0
        \end{array}\right)
        \quad 
        Q=\left(\begin{array}{ccc}
        3/4-\epsilon & 1/4+\epsilon\\
        3/4-\epsilon & 1/4+\epsilon \\
        1/4+\epsilon & 3/4-\epsilon\\
        1/4+\epsilon & 3/4-\epsilon
        \end{array}\right),
        \end{align*}
        \begin{align*}
            \sT_1=\left(\begin{array}{cccc}
            1/3 & 1/2 & 1/6 & 0 \\
            0 & 1/3 & 1/2 & 1/6 \\
            1/2 & 1/3 & 0 & 1/6 \\
            1/3 & 1/3 & 1/3 & 0
            \end{array}\right)
        \end{align*}
\end{example}
For this example, please note that 
$\delta(Q)$ is greater than $1/2$, and the diameter 
$D$ is equal to 1. We can choose 
$\alpha$ to be 1. Hence, according to Theorem \ref{main}, 
the Average Cost Optimal Equation (ACOE) has a solution.

Alternatively, if we can select $\alpha$ to be less than 1, a more relaxed condition on 
$\delta(Q)$ would suffice. If we can choose 
$\alpha$ to be less than $2/3$, 
without any other constraints on the observation kernel, 
we can assert that the Average Cost Optimal Equation (ACOE) 
has a solution. As we will see in the literature review section,
these conditions are new.

The second example will be for the continuous state space case.

\begin{example}
    Let $\X=[0, 2]$, $\Y=\{0, 1\}$,
        $\U=[0, 12]$,
        let the transition kernel
        $\T(.\mid x,u)=\mathrm{Unif}(0,min(2,1+((x+u)/7)))$, where 
        $\mathrm{Unif}$ stands for uniform distribution,
        let measurement kernel $Q(x)=\lfloor x \rfloor$ and 
        let cost function $c(x,u)=x+u$.
\end{example}
        For this example, 
        we observe that $\delta(Q)=0$. 
        Furthermore, considering any control input 
        $u$ in the set $\U$ and any states $x$ and $x'$ 
        within the space $\X$ such that $x' < x$, 
        we can derive the following bound:
$$
\begin{aligned}
&\left\| \T(\cdot | x, u) - \T(\cdot | x', u) \right\|_{TV} \\
&\leq \left\| \mathrm{Unif}(0,1+(x'+u)/7)-\mathrm{Unif}(0,1+(x+u)/7) \right\|_{TV}\\
&=2\left( 1 - \frac{1 + \frac{x'+u}{7}}{1+\frac{x + u }{7}}\right) \\
&= 2\frac{x - x'}{x + u + 7} \\
& < \frac{2}{7} d(x, x').
\end{aligned}
$$

We have $D = \sup d(x,x') = 2$, 
and by choosing $\alpha = \frac{2}{7}$ and $K_1 = 1$, 
we can set $K_2$ to $6/7$. 
This choice ensures that we satisfy 
the conditions outlined in Assumption 
\ref{main_assumption}. Consequently, 
we can directly apply Theorem \ref{main} 
to this specific example.

\begin{example}\label{ContE}
Consider $\X=[0,1]$, $\U=[-p,p]$ and the cost function $c(x,u)=x-u$. Define the transition kernel $\T(.|x,u)=\bar{N}(x+u,\sigma^2)$.
Here, $\bar{N}(\mu,\sigma^2)$ denote truncated version 
of $N(\mu,\sigma^2)$, where the support is restricted to $[0,1] \subset \mathbb{R}$.
Its probability density function $f$ is given by
$$f(x ; \mu, \sigma)=\frac{1}{\sigma} 
\frac{\varphi\left(\frac{x-\mu}{\sigma}\right)}
{\Phi\left(\frac{1-\mu}{\sigma}\right)-\Phi\left(\frac{0-\mu}{\sigma}\right)}$$
Here, $\varphi(\cdot)$ is the probability density function 
of the standard normal distribution and 
$\Phi(\cdot)$ is its cumulative distribution function.

For any $0\leq x < y \leq 1$:
$
\frac{\lVert \T(.|y,u)-\T(.|x,u) \rVert _{TV}}{y-x}\leq
\frac{\sqrt{2}}{\sigma \sqrt{\pi}}.
$
This shows that the transition kernel $\T$ satisfies Assumption \ref{main_assumption}-\ref{regularity} with $\alpha = \frac{\sqrt{2}}{\sigma \sqrt{\pi}}$.
If we choose $\sigma>3/\sqrt{2\pi}$, Assumption \ref{main_assumption}-\ref{dobr_coef} is satisfied for every observation channel. 
Under this condition, by Theorem \ref{main}, the 
ACOE has a solution.
\end{example}

\subsection{Literature Review and Comparison}

In the following we present a detailed literature 
review and also provide a 
comparison with our results. 
We will in particular highlight that
 many of the results in the 
 literature are not easy to verify, 
 which also serves to demonstrate that the problem is a challenging one. 

References \cite{platzman1980optimal,fernandez1990remarks,runggaldier1994approximations,hsu2006existence} 
study the average-cost control problem under 
the assumption that the state space is finite; 
they provide reachability type conditions for the 
belief kernels. Reference \cite{hsu2006existence} also provides a detailed literature review. An additional line of argument, via simulation and coupling, was introduced in \cite{Bor00,Bor03,borkar2004further,Bor07}. Under certain strong continuity requirements 
of the transition kernel, \cite{StettnerACOE2019} 
extends the positivity condition to Polish spaces.
We also note that, via the weak Feller property of non-linear filters, 
the convex analytic methods can also be utilized \cite[Theorem 1.2]{anotherLookPOMDPs}, 
though the dependence on the initial condition is a limitation.


Borkar and Budhiraja \cite{borkar2004further} consider $\X$, $\Y$, and $\U$ as Polish spaces, 
where $\X$ is a finite-dimensional Euclidean space (not necessarily compact, unlike ours)
and $\U$ is a compact space. In their setup, 
unlike ours, the observation measurement 
depends on the previous time's state and action. For $(x, u) \in \mathbb{X} \times \mathbb{U}$, 
the transition probability measure 
$p(x, u, dz, dy) \in \Pf(\mathbb{X} \times \mathbb{Y})$ 
is defined.  \cite{borkar2004further} assumes that $p$ is continuous in a strong sense: Let $\lambda$ denote the Lebesgue measure on $\mathbb{X}$ and let $\eta \in \Pf(\mathbb{Y})$, assume a density 
function $\varphi(x, u, z, y)$ on 
$\mathbb{X} \times \mathbb{U} \times \mathbb{X} \times \mathbb{Y}$ 
such that $p(x, u, dz, dy) = \varphi(x, u, z, y) \lambda(dz) \eta(dy)$, 
with $\varphi(\cdot) > 0$. It is assumed that $\varphi(x, u, z, y)$ is continuous. Note that by Scheff\'e's lemma, this implies that the transition kernel as well as the measurement kernel are continuous in total variation. The update rule is then given as
\begin{align*}
    Pr\left(X_{n+1} \in A, Y_{n+1} \in A^{\prime} \mid X_{[0,n]},Y_{[0,n]},U_{[0,n]} \right) \\
    =\int_{A^{\prime}} \int_A \varphi\left(X_n, U_n, z, y\right) \lambda(\mathrm{d} z) \eta(\mathrm{d} y).
\end{align*}
Furthermore, to facilitate a coupling argument, \cite{borkar2004further} additionally assumes
\begin{align}
\int \bar{\varphi}\left(x, u, x^{\prime}
\right)\left(\frac{\varphi\left(x, u, x^{\prime},
 y\right)}{\bar{\varphi}\left(x, u, x^{\prime}\right)}
 \right)^{1+\varepsilon_0} \lambda\left(d x^{\prime}\right) 
 \eta(d y)<\infty \label{bound123}
\end{align}
for some $\varepsilon_0 > 0$, where 
$\bar{\varphi}(x, u, x') = \int \varphi(x, u, x', y) \eta(dy)$.
Finally, the following Lyapunov-type 
assumption is assumed (which always holds when 
$\X$ is compact).
\begin{assumption}\label{Borkar0}
There exist $\V \in \C(\X)$ and
$\hat{\V} \in \C(\X)$ 
satisfying
$\lim _{\|x\| \rightarrow \infty} \V(x)=\infty$ and
$\lim _{\|x\| \rightarrow \infty} 
\hat{\V}(x)=\infty$, 
and under any
wide sense admissible $\left\{Z_n\right\}$,
\begin{align*}
E\left[\hat{\V}\left(X_{n+1}\right) \mid \F_n\right]-\hat{\V}\left(X_n\right) 
\leq-\V\left(X_n\right)+\hat{C} I_{\hat{B}}\left(X_n\right),
\end{align*}
\begin{align*}
    \limsup _{n \rightarrow \infty} \frac{E\left[\hat{\V}\left(X_n\right)\right]}{n}=0,
\end{align*}
where $\hat{C}>0$,
$\hat{B} = \{x \in S:\|x\| \leq \hat{R}\}$ 
for some $\hat{R}>0$ and 
$\F_n=\sigma(X_{[0,n]},Y_{[0,n]},U_{[0,n]})$.
\end{assumption}
\begin{theorem}[Theorem 4.1 of \cite{borkar2004further}] 
    For the system described above and under the 
    mentioned assumptions as well as Assumption \ref{Borkar0}, 
    a solution to the average cost optimality equation 
    (ACOE) exists. This implies the existence of an 
    optimal control policy, and the optimal cost remains 
    constant for every initial state.
\end{theorem}

Compared to this result, we present complementary 
conditions which are explicit and testable. 
Additionally, for the compact setup, we do not have continuity assumptions on the measurement kernels, and the condition (\ref{bound123}) is not needed.  
Our paper looks to be the first one where 
contraction properties of filter kernels is utilized. 

In \cite{platzman1980optimal}, 
Platzman addresses the finite-state 
average-cost control problem with finite state, 
observation, and action spaces, under restrictive reachability, subrectangularity, and detectability conditions.

Runggaldier and Stettner in \cite{runggaldier1994approximations} consider $\X$ and $\Y$ as finite and $\U$ as compact. They prove that under 
the following positivity condition \cite{runggaldier1994approximations}, ACOE has a bounded solution:
\begin{assumption}[\cite{runggaldier1994approximations}]\label{Stettner}
\[\inf_{i, j \in \X} \inf_{u, u^{\prime} \in \U} \inf_{\left\{C \in \B(\X):[\T(C|i, u)]>0\right\}} \frac{\T\left(C|j,u^{\prime}\right)}{T(C|i, u)} > 0.\]
\end{assumption}

We note that Example \ref{ex1} does not meet this condition.

Borkar \cite{Bor00} employs a coupling argument under the following assumption.

\begin{assumption}[\cite{Bor00} Assumption A]\label{Borkar0coupling}
    There exist constants $K_0 \in \mathbb{R}^+$ and $\delta \in (0, 1)$ such that $\sup_{i, j} \sup_{\gamma} P(\tau > n | X'_0 = i, X_0 = j) \leq K_0 \delta^n$ holds for all $n \geq 0$.
    This supremum is taken over all wide-sense admissible policies, and 
    $\tau$ represents the coupling time, i.e., 
    $\tau=\min \{n:X'_n =X_n\}$.
\end{assumption}
However, verifying this assumption is not simple due to the necessity of taking the supremum over all wide-sense admissible policies (which is a strict superset of deterministic admissible policies).

Hsu, Chuang, and Arapostathis \cite{hsu2006existence} consider
a finite $\mathbb{X}$ and $\mathbb{Y}$ and a compact $\mathbb{U}$. 
They provide two different sets of assumptions under which ACOE has a solution.


\begin{assumption}[Assumption 2 of \cite{hsu2006existence}]\label{Ari1}
$\mathcal{Z}_{\epsilon} = \{\mu \in \mathcal{Z}: \mu(x) > \epsilon \quad \forall x \in \X\}$. 
There exist constants $\epsilon > 0$, $k_0 \in \mathbb{N}$, and $\alpha < 1$ such that if $z_*(\beta) \in \arg \min_{z \in \Z} J^*_{\beta}(z)$, then for each $\beta \in [\alpha, 1)$ we have
$$
\max_{1 \leq k \leq k_0} \mathbb{P}_{z_*(\beta)}^{\gamma_{\beta}}(Z_k \in \mathcal{Z}_{\epsilon}) \geq \epsilon,
$$
where $Z_k$ is the filter process and $\gamma_{\beta}$ is the optimal policy for the $\beta$-discounted horizon cost problem.
\end{assumption}
Because checking this assumption can be quite complex, an alternative easier-to-verify assumption is provided

\begin{assumption}[\cite{hsu2006existence}]\label{Ari5}
    There exist $k \geq 1$, and $\Delta > 0$ such that, for all admissible $\{U_t\}$
    $$
    \mathbb{P}(X_k = j | X_0 = i, U_{t-1}, Y_t, 1 \leq t \leq k) \geq \Delta \text{ for all } i, j \in \X.
    $$
\end{assumption}

For the second set of assumptions, 
the conditions are relaxed to some extent 
when the cost function $c$ is continuous.
%
%

\begin{assumption}[\cite{hsu2006existence}]\label{Ari4}
    There exist $k \geq 1$, and $\Delta>0$ such that, 
    for all $y^k \in \Y^k$ and $u^k \in \U^k$,
        $$
        P_{i j}\left(y^k|u^k\right) \geq \Delta 
        \sum_{\ell \in \X} P_{\ell j}\left(y^k| u^k\right) \quad \forall i, j \in \X,
        $$
        where 
        $P_{i j}\left(y^k \mid u^k\right)=\mathbb{P}\left(X_k=j, Y_{[1, k]}=y^k \mid X_0=i, U_{[0, k-1]}=u^k\right)$.
\end{assumption}

\begin{theorem}[Theorem 11 of \cite{hsu2006existence}] 
    Under either Assumption \ref{Ari1} or Assumption \ref{Ari4}, 
    the Average Cost Optimal Equation (ACOE) possesses a bounded solution.
\end{theorem}

Finally, Stettner, in \cite{StettnerACOE2019}, extends 
the results of \cite{runggaldier1994approximations} 
to Polish state spaces and observation spaces, 
encompassing both nondegenerate and degenerate observations: For $u_1, u_2 \in \U$ and $\mu, \nu \in P(\X)$ 
define $$\lambda\left(u_1, u_2, \mu, \nu\right):=\inf _{\{C: \T(C|\mu, u_1)>0\}} \frac{\T(C|\nu, u_2)}{T(C|\mu, u_1)}.$$ 
Then define $\lambda(\mu, \nu)=\inf _{u \in U} \lambda(u, u, \mu, \nu)$.

\begin{assumption}[\cite{StettnerACOE2019}]\label{Stettner2}
$\lambda\left(u^1, u_n^2, \mu, \nu_n\right) \rightarrow 1$ and 
$\lambda\left(u_n^2, u^1, \nu_n, \mu\right) \rightarrow 1$ when 
$u_n^2 \rightarrow u^1$ and $\nu_n \Rightarrow \mu$.
Similarly, assume that both $\lambda\left(\nu_n, \mu\right)$ and $\lambda\left(\mu, \nu_n\right)$ converge to 1 as $\nu_n \Rightarrow \mu$
(where $\Rightarrow$ denotes weak convergence of probability measures).
\end{assumption}

\begin{theorem}[Theorem 5.6 of \cite{StettnerACOE2019}]
    Let $\X$ and $\Y$ be Polish spaces, $\U$ be a 
    compact space, and $c$ be a continuous and bounded function. 
    Under Assumptions \ref{Stettner} and \ref{Stettner2}, 
    the Average Cost Optimality Equation (ACOE) admits a bounded solution.
\end{theorem}

In \cite{yu2008near}, the authors study near optimality of finite window policies for average cost problems where the state, action and observation spaces are finite; under the condition that the liminf and limsup of the average cost are equal and independent of the initial state, the paper establishes the near-optimality of (non-stationary) finite memory policies. Here, a concavity argument building on \cite{Feinberg2} (which becomes consequential by the equality assumption) and the finiteness of the state space is crucial. The paper shows that for any given $\epsilon>0$, there exists an $\epsilon$-optimal finite window policy. 

With this review, we have both summarized some existing key studies, and also highlighted that our paper presents accessible and testable conditions, compared with most of the literature reviewed above. 

In our paper, we present a contraction based analysis for non-linear filters, which is a novel contribution and which we expect to have significant consequences for learning theoretic and approximation results.

\section{Proof of the Main Theorem}

Recently, \cite{kara2020near} presented the following regularity results for controlled filter processes:

\begin{theorem}[\cite{kara2020near}, Theorem 7-i, Theorem 7-iv]\label{kara2020near1}
Assume that $\mathbb{X}$ and $\mathbb{Y}$ are Polish spaces. 
\begin{enumerate}
    \item[i.] If Assumption \ref{main_assumption}-\ref{regularity} is fulfilled, then we have
    \[
    \rho_{B L}\left(\eta(\cdot \mid z, u), \eta\left(\cdot \mid z^{\prime}, u\right)\right) \leq 
    3\left(1+\alpha \right) \rho_{B L}\left(z, z^{\prime}\right)
    \]
    for any $z,z'\in {\mathcal P}(\mathbb{X})$ and $u\in \U$.
    \item[ii.] 
    Without any assumptions
    $$
        \rho_{BL}\left(\eta(\cdot|z,u),\eta(\cdot|z',u)\right)\leq (3-2\delta(Q))(1-\tilde{\delta}(\mathcal{T}))\|z-z'\|_{TV}
    $$
    for any 
    $z,z'\in {\mathcal P}(\mathbb{X})$ and $u\in \U$,
    where $\tilde{\delta}(\mathcal{T}):=\inf _{u \in \mathbb{U}} \delta(\mathcal{T}(\cdot \mid \cdot, u))$
    
\end{enumerate}
\end{theorem}

Building on these results and their proof method, the following result follows from \cite{demirci2023geometric}, Theorem 5.1] which considered the control-free setup. A proof sketch is given in the appendix.
\begin{theorem}\label{ergodicity}
    Assume that $\mathbb{X}$ and $\mathbb{Y}$ are Polish spaces. 
    If Assumption \ref{main_assumption}-\ref{compactness} and Assumption \ref{main_assumption}-\ref{regularity} are 
    fulfilled, then we have
    $$
    W_{1}\left(\eta(\cdot \mid z_0, u), \eta\left(\cdot \mid z_0^{\prime},u\right)\right) 
    \leq \left(\frac{\alpha D (3-2\delta(Q))}{2}\right) W_{1}\left(z_0, z_0^{\prime}\right) .
    $$
    for all $z_0,z_0' \in \mathcal{Z}$, $u \in \mathbb{U}$.
\end{theorem}
As we assume that $\mathbb{X}$ is compact, under the $W_1$ (that is, the $1$-Wasserstein or Kantorovich -Rubinstein) metric, $\mathcal{Z}$ is compact.

Under Assumption \ref{main_assumption}-\ref{CostLipschitz}, 
we have that $\tilde{c}$ defined in (\ref{tildecost}) is Lipschitz continuous, since
\begin{align}\label{c_tilde_cont}
&    |\tilde{c}(z,u) - \tilde{c}(z',u)| \nonumber \\
&    =    \left|\int_{\mathbb{X}}c(x,u)z(dx) - 
    \int_{\mathbb{X}}c(x,u)z'(dx)\right|
    \leq K_1 W_{1}\left(z, z^{\prime}\right).
\end{align}

\begin{theorem}\cite[Theorem 2]{KSYWeakFellerSysCont} \label{weakFeller}
   Under Assumption \ref{main_assumption}-\ref{totalvar}, 
    the transition probability 
    $\eta(\cdot \mid z, u)$ of the filter process is 
    weakly continuous in $(z, u)$.
\end{theorem}
\noindent We also highlight that additional 
recent results are available on the weak Feller property, such as those in 
\cite{FeKaZg14}; however, as there are no restrictions on $Q$, 
this result \cite[Thm.2]{KSYWeakFellerSysCont} is relevant here.

\begin{lemma}\label{keyRegLemF}
Under Assumption \ref{main_assumption},
for any $\beta$ the value function $J_\beta^*$ 
is 
Lipschitz continuous with coefficient $K$, 
where $K=\frac{K_1}{1-\beta K_2}$.
\end{lemma}

\begin{proof} We adopt the approach in the proof of \cite[Theorem 4.37]{SaLiYuSpringer}. Let $f \in \operatorname{Lip}(\mathcal{Z}, k)$ for some $k>0$. 
Then $g=\frac{f}{k} \in \operatorname{Lip}(\mathcal{Z}, 1)$ 
and therefore for all $u \in \mathrm{U}$ and 
$z, y \in \mathcal{Z}$ we have
\begin{align}
& \left|\int_{\mathcal{Z}} f(x) \eta(d x \mid z, u)-
\int_{\mathcal{Z}} f(x) \eta(d x \mid y, u)\right| \\
& =k\left|\int_{\mathcal{Z}} g(x) \eta(d x \mid z, u)-
\int_{\mathcal{Z}} g(x) \eta(d x \mid y, u)\right| \\
& \leq k W_1(\eta(\cdot \mid z, u), \eta(\cdot \mid y, u)) 
\leq k K_2 W_1(z, y),
\end{align}
by Theorem \ref{ergodicity}. 
Let $T$ be the Bellman optimality operator. 
$$
(Tf)(x) = \min_{u \in U} \left\{ c(x, u) + \beta \int_{y \in \mathcal{Z}} \eta(dy | x, u) f(y) \right\} 
$$
We know that $J^*_\beta$ satisfy $T J^*_\beta= J^*_\beta$ \cite[Lemma 4.2.6]{HernandezLermaMCP}. 
$T$ is a contraction, so by Banach fixed point theorem $T^n f$ converges to $J^*_\beta$.
$$
\begin{aligned}
\mid T f(z) & -T f(y) \mid \\
& \leq \max _{u \in \mathrm{U}}
\bigg\{|\tilde{c}(z, u)-\tilde{c}(y, u)|  \nonumber \\
& +
\beta\left|\int_{\mathbf{Z}} f(x) \eta(d x \mid z, u)
-\int_{\mathbf{Z}} f(x) \eta(d x \mid y, u)\right|\bigg\} \\
& \leq K_1 W_{1}(z, y)+\beta k K_2 W_{1}(z, y)=
\left(K_1+\beta k K_2\right) W_{1}(z, y) =: M_1 W_{1}(z, y).
\end{aligned}
$$
By induction we have for all $n \geq 2$
$$
T^n f \in \operatorname{Lip}\left(\Z, M_n\right),
$$
where $M_{n} = K_1 + \beta K_2 M_{n-1}$ and thus $M_n=K_1 \sum_{i=0}^{n-1}\left(\beta K_2\right)^i+k\left(\beta K_2\right)^n$. 
Then, the sequence $M_n$ monotonically converges to
 $\frac{K_1}{1-\beta K_2}$ 
since $K_2<1$. Hence, 
$T^n f \in \operatorname{Lip}\left(\sZ, \frac{K_1}{1-\beta K_2}\right)$ 
for all $n$, and therefore, 
$J_\beta^* \in \operatorname{Lip}\left(\sZ, \frac{K_1}{1-\beta K_2}\right)$ 
since it is closed with respect to the sup-norm. Taking $k \leq \frac{K_1}{1-\beta K_2}$, we certify that the fixed point satisfies the desired Lipschitz continuity. 
\end{proof}
For any $\beta \in (0,1)$, $J^*_\beta \in 
\operatorname{Lip}\left(\sZ, \frac{K_1}{1-K_2}\right)$.
Therefore, for any $z_0\in \mathcal{Z}$
\begin{align}\label{boundedWas}
h_\beta(z) = J^*_{\beta}(z) - J^*_{\beta}(z_0) \leq 
\frac{K_1}{1 - K_2} W_1(z,z_0)\leq \frac{K_1\cdot D}{1-K_2}.
\end{align}

In view of the results above, we now introduce a crucial auxiliary result 
that will play a pivotal role in establishing our main result:

\begin{assumption}\label{lect}
\noindent
\begin{enumerate}
\item \label{bddcostfunction}
The one stage cost function $\tilde{c}$ is bounded and continuous.
\item \label{contK}
The stochastic kernel $\eta(\cdot \mid x, u)$ is 
weakly continuous in $(x, u) \in \Z \times \mathbb{U}$, 
i.e., if $\left(x_k, u_k\right) \rightarrow(x, u)$, 
then $\eta\left(\cdot \mid x_k, u_k\right) \rightarrow$ 
$\eta(\cdot \mid x, u)$ weakly.
\item 
$\mathbb{U}$ is compact.
\item
$\Z$ is $\sigma$-compact, that is, 
$\Z=\cup_n S_n$ where $S_n \subset S_{n+1}$ and each $S_n$ is compact.
\setcounter{tempcounter}{\value{enumi}}
\end{enumerate}
There exists $\alpha \in(0,1)$ and $N \geq 0$, and a state $z_0 \in \Z$ such that,
\begin{enumerate}
    \setcounter{enumi}{\value{tempcounter}}
\item \label{bddh}
$-N \leq h_\beta(z) \leq N$ for all $z \in \Z$ and $\beta \in[\alpha, 1)$, where
\begin{align}\label{relativeValue}
h_\beta(z)=J^*_\beta(z)-J^*_\beta\left(z_0\right).
\end{align}
\item \label{equic}
The sequence $\left\{h_{\beta(k)}\right\}$ 
is equicontinuous, where $\{\beta(k)\}$
is a sequence of discount factors converging to 
1 which satisfies $\lim _{k \rightarrow \infty}(1-\beta(k)) J_{\beta(k)}^*(z)=\rho^*$ 
for all $z \in \Z$ for some $\rho^* \in[0, L]$.
\end{enumerate}
\end{assumption}

\begin{lemma}\cite[Theorem 7.3.3]{yuksel2020control}\label{exisACOE}
    Under Assumption \ref{lect}, 
    a solution to the average cost optimality 
    equation (ACOE) exists, 
    leading to the existence of an optimal 
    control policy, and optimal cost is constant for 
    every initial state.
\end{lemma}
\begin{proof}
    See Appendix \ref{Acoe_proof}.
\end{proof}

\noindent We note that a similar result, with slightly stronger continuity conditions (not applicable to our setup) under \cite[Assumption 4.2.1]{HernandezLermaMCP}, can be found in \cite[Theorem 5.5.4]{HernandezLermaMCP}. Using this auxiliary result, we are now ready to present the proof of the main theorem.

\begin{proof}[Proof of Theorem \ref{main}]
First, we will show that if a partially 
observable Markov decision process (POMDP) 
satisfies Assumption \ref{main_assumption}, 
then the corresponding belief MDP satisfies Assumption \ref{lect}.

Assumption \ref{lect}-\ref{bddcostfunction} is 
valid due to equation (\ref{c_tilde_cont}). 
Assumptions \ref{lect}-3 and 4 hold because $\U$ 
and $\Z$ are compact. Assumption \ref{lect}-\ref{contK} 
follows from Theorem \ref{weakFeller}, and 
Assumption \ref{lect}-\ref{bddh} follows from 
(\ref{boundedWas}).

By inequality (\ref{boundedWas}), we know
\[h_\beta(z) = J^*_{\beta}(z) - J^*_{\beta}(z_0) \leq 
\frac{K_1}{1 - K_2} W_1(z,z_0)\]
for all $\beta \in (0,1)$. Therefore, $h_\beta$ is 
equicontinuous for all $\beta \in (0,1)$. 
Since this condition holds for all subsequences, 
Assumption \ref{lect}-\ref{equic} is satisfied.

Thus, if a POMDP satisfies Assumption \ref{main_assumption}, 
then the belief MDP satisfies Assumption \ref{lect}. 
All conditions of Lemma \ref{exisACOE} are satisfied, 
and the proof is completed using Lemma \ref{exisACOE}.
\end{proof}

\section{Implications 
for Approximations, Robustness
and Learning}\label{impl}

In this section, we discuss several implications 
of our results on the existence
of a solution to the average
cost optimality equation and Wasserstein regularity.

First, we establish robustness; studying how an optimal 
policy developed for an incorrect initial prior performs when applied to the correct 
initial distribution in the context of the average 
cost criteria. Under certain conditions, we find 
that the average cost optimization problem 
is robust to errors in initialization.

We then study approximate optimality through 
the quantization of the belief space, or $\Pf(\X)$. 
This method yields a near-optimal 
policy for the original problem. We then present an alternate 
method to construct a finite Markov Decision 
Process (MDP). This is achieved by replacing 
the complete observable Markov process with 
the most recent N observations and actions. 
Within this finite state MDP framework, 
we apply Q-learning to obtain a near-optimal policy.

\subsection{Robustness to Incorrect Priors for Discounted and Average Cost Criteria}\label{robustness}
In this section, we discuss the implications of our 
results on robustness. Let us first formally define the 
robustness problem. 
We have already introduced $J^*$ and $J_{\beta}^*$ in 
the first section. 

An optimal control policy $\gamma^\mu$ for 
a given prior $\mu$ is a policy 
that achieves the lowest expected 
cost over all admissible control policies:
$$
J\left(\mu, \gamma^\mu\right)=\inf _{\gamma \in \Gamma} J(\mu, \gamma)=J^*(\mu)
$$

Consider a scenario where a controller incorrectly assumes the system's prior to be $\nu$, while in reality, it is $\mu$. 
In this case, the controller would implement the policy  $\gamma^\nu$, optimal for $\nu$, but this results in an expected cost of $J\left(\mu, \gamma^\nu\right)$. If the controller
used the correct policy, the cost could have been $J^*(\mu)$. In studying robustness, we are interested in this difference:
$$
\begin{gathered}
J_\beta\left(\mu, \gamma^\nu_\beta\right)-J_\beta^*(\mu), \\
J\left(\mu, \gamma^\nu\right)-J^*(\mu)
\end{gathered}
$$
for the discounted cost problem and the average 
cost problem. Here $\gamma^\nu_\beta$ is an optimal
policy for the initial $\nu$ in the context of discounted cost criteria,
while $\gamma^\nu$ is an optimal
policy for the initial $\nu$ in the context of average cost criteria.

In the literature on POMDPs, there are upper
 bounds on this error cost for discounted 
 cost problem \cite{KYSICONPrior}, 
 studies on its relationship with 
 filter stability, and attempts to 
 find uniform upper bounds for average 
 cost problem \cite{MYRobustControlledFS}. 
 There is also work considering how 
 robustness is affected under 
 different transition kernels for 
 discounted cost problems \cite{kara2020robustness}. The following builds on \cite[Theorem 3.2]{KYSICONPrior}.


\begin{theorem}\cite[Theorem 3.8]{MYRobustControlledFS},\cite[Theorem 3.2]{KYSICONPrior}
Assume the cost function $c$ is bounded, nonnegative, and measurable. Let $\gamma^\nu$ be the optimal control policy designed with respect to a prior $\nu$. Then we have
$$
J\left(\mu, \gamma^\nu\right)-J^*(\mu) \leq 2\|c\|_{\infty}\|\mu-\nu\|_{T V}
$$
\end{theorem} 
However, by utilizing filter stability, the bound can be refined, as we will observe in the following. 
\begin{defn}
    A filter process is said to be stable in the sense of total variation in expectation with respect to a policy $\gamma$ if for any measure $\nu$ with $\mu \ll \nu$ we have 
    \newline $\lim _{n \rightarrow \infty} E^{\mu, \gamma}\left[\left\|\pi_n^{\mu, \gamma}-\pi_n^{\nu, \gamma}\right\|_{T V}\right]=0$.
    \newline The filter is universally stable in total variation in expectation if it holds with respect to every admissible policy $\gamma \in \Gamma$.
\end{defn}

For the average cost criteria, under Assumption \ref{main_assumption},
asymptotic filter stability is sufficient for robustness, as the following shows:

\begin{theorem}[\cite{MYRobustControlledFS}, Theorem 3.9]
Assume the cost function $c$ is bounded, nonnegative, and measurable and assume the filter is universally stable in total variation in expectation. Consider the span semi-norm:
$$
\left\|J^*\right\|_{s p}:=\sup _{\mu_1 \in \mathcal{P}(\mathcal{X})} J^*\left(\mu_1\right)-\inf _{\mu_2 \in \mathcal{P}(\mathcal{X})} J^*\left(\mu_2\right)
$$
then we have
$$
J\left(\mu, \gamma^\nu\right)-J^*(\mu) \leq\left\|J^*\right\|_{s p} .
$$
In particular, if $\left\|J^*\right\|_{s p}=0$, then the average cost optimization problem is completely robust to initialization errors.
\end{theorem} 
We note that in \cite{MYRobustControlledFS}, the existence of an optimal solution was not shown; our implication also presents an existence result for optimal policies. Under Assumption \ref{main_assumption}, we show that $\left\|J^*\right\|_{s p}=0$ 
(Theorem \ref{main}), meaning if the filter is 
universally stable in total variation in expectation, 
then the average cost optimization problem is 
completely robust to initialization errors. 
\begin{corollary}
    Under Assumption \ref{main_assumption} and 
    assuming universal filter stability in total 
    variation in expectation, we have that
    $$
    J\left(\mu, \gamma^\nu\right)=J^*(\mu) \quad \forall \mu,\nu \in \Z .
    $$
\end{corollary}
\begin{note}
As mentioned in Note \ref{Curtis_exp}, 
\cite{mcdonald2020exponential}, Theorem 4.1, 
demonstrates that if the condition 
$(1-\tilde{\delta}(\T))(2-\delta(Q)) < 1$ is met, 
then the filter is universally stable in total
variation in expectation.
\end{note}

In addition to our results on the average cost optimality equation,
our analysis is also consequential for the discounted cost criterion
and offering a improvement to \cite[Theorem 3.10]{MYRobustControlledFS}. 
Utilizing the span semi-norm bound derived in our analysis 
of the average cost optimality equation, and by 
modifying the proof of \cite[Theorem 3.10]{MYRobustControlledFS}, 
along with individually bounding equations (1.8), (1.9), and (1.10) 
from the same source, we arrive at the following theorem:
\begin{theorem}
Under Assumption \ref{main_assumption} and the condition 
$\bar{\alpha}:=(1-\tilde{\delta}(T))(2-\delta(Q)) < 1$, we have the following robustness bound:
$$J_{\beta}(\mu,\gamma^{\nu}_\beta) - J^*_{\beta}(\mu) 
\leq \inf_{n \in \mathbb{N}}
\left( \|c\|_{\infty}\frac{(1 - \beta^{n})}{1-\beta} + 
\beta^n \frac{D K_1}{1 - K_2\beta} + 4 \frac{\|c\|_{\infty}}{1-\beta}(\bar{\alpha} \beta)^n\right).$$
\end{theorem}
\begin{proof}
    A brief proof is provided here. 
    \begin{align}
    & \nonumber J_\beta\left(\mu, \gamma^\nu_\beta\right)-J_\beta^*(\mu)
    \\ \nonumber & = E^{\mu, \gamma^\nu_\beta}\left[\sum_{i=0}^{n-1} 
    \beta^i c\left(X_i, U_i\right)\right]-E^{\mu, \gamma^\mu_\beta}
    \left[\sum_{i=0}^{n-1} \beta^i c\left(X_i, U_i\right)\right] 
    \\ \nonumber & +\beta^n\left(E^{\mu, \gamma^\nu_\beta}\left[J_\beta
    \left(\pi_{n-}^{\mu, \gamma_\beta^\nu}, \gamma_\beta^{\nu^{\prime}}\right)
    \right]-E^{\mu, \gamma_\beta^\mu}\left[J_\beta^*
    \left(\pi_{n-}^{\mu, \gamma_\beta^\mu}\right)\right]\right) 
    \\ \nonumber & = E^{\mu, \gamma_\beta^\nu}\left[\sum_{i=0}^{n-1} 
    \beta^i c\left(X_i, U_i\right)\right]-E^{\mu, \gamma_\beta^\mu}
    \left[\sum_{i=0}^{n-1} \beta^i c\left(X_i, U_i\right)\right] 
    \\ \nonumber & +\beta^n\left(E^{\mu, \gamma_\beta^\nu}\left[J_\beta
    \left(\pi_{n-}^{\mu, \gamma_\beta^\nu}, \gamma_\beta^{\nu^{\prime}}
    \right)+J_\beta^*\left(\pi_{n-}^{\mu, \gamma_\beta^\nu}\right)
    -J_\beta^*\left(\pi_{n-}^{\mu, \gamma_\beta^\nu}\right)\right]-
    E^{\mu, \gamma_\beta^\mu}\left[J_\beta^*
    \left(\pi_{n-}^{\mu, \gamma_\beta^\mu}\right)\right]\right)
    \\ \label{transient_cost} &=E^{\mu, \gamma_\beta^\nu}\left[\sum_{i=0}^{n-1} \beta^i c\left(X_i, U_i\right)\right]-
     E^{\mu, \gamma_\beta^\mu}\left[\sum_{i=0}^{n-1} \beta^i c\left(X_i, U_i\right)\right]
    \\ \label{strategic_m_cost} & +\beta^n\left(E^{\mu, \gamma_\beta^\nu}\left[J_\beta^*
    \left(\pi_{n-}^{\mu, \gamma_\beta^\nu}\right)\right]
    -E^{\mu, \gamma_\beta^\mu}\left[J_\beta^*
    \left(\pi_{n-}^{\mu, \gamma_\beta^\mu}\right)\right]\right)
    \\ \label{approximation_cost} & +\beta^n\left(E^{\mu, \gamma_\beta^\nu}\left[J_\beta
    \left(\pi_{n-}^{\mu, \gamma_\beta^\nu}, \gamma_\beta^{\nu^{\prime}}\right)
    -J_\beta^*\left(\pi_{n-}^{\mu, \gamma_\beta^\nu}\right)\right]\right),
    \end{align}
    where $\nu^{\prime}=\pi_{n-}^{\nu, \gamma_\beta^\nu}$ and 
    $\gamma_\beta^{\nu^{\prime}}$ is an optimal discounted policy
    for $\nu^{\prime}$.
As in \cite{MYRobustControlledFS}, we can refer to 
these three components as transient cost, 
strategic measure cost, and approximation cost, respectively.  The transient cost (\ref{transient_cost}) is upper bound by
$$
E^{\mu, \gamma_\beta^\nu}\left[\sum_{i=0}^{n-1} \beta^i c\left(x_i, u_i\right)\right]-E^{\mu, \gamma_\beta^\mu}\left[\sum_{i=0}^{n-1} \beta^i c\left(x_i, u_i\right)\right] \leq\|c\|_{\infty} \sum_{i=0}^{n-1} \beta^i=\|c\|_{\infty}\left(\frac{1-\beta^n}{1-\beta}\right).
$$
The strategic measure cost (\ref{strategic_m_cost}) is upper bound by
$$
\beta^n\left(E^{\mu, \gamma_\beta^\nu}\left[J_\beta^*\left(\pi_{n-}^{\mu, \gamma_\beta^\nu}\right)\right]-
E^{\mu, \gamma_\beta^\mu}\left[J_\beta^*\left(\pi_{n-}^{\mu, \gamma_\beta^\mu}\right)\right]\right) 
\leq \beta^n\left\|J_\beta^*\right\|_{s p}
\leq \beta^n \frac{K_1}{1-\beta K_2} D
$$
because of Lemma \ref{keyRegLemF}.
The approximation cost (\ref{approximation_cost}) is upper bound by
$$
\beta^n\left(E^{\mu, \gamma_\beta^\nu}\left[J_\beta\left(\pi_{n-}^{\mu, \gamma_\beta^\nu}, \gamma_\beta^{\pi_{n-}^{\nu, \gamma_\beta^\nu}}\right)-J_\beta^*\left(\pi_{n-}^{\mu, \gamma_\beta^\nu}\right)\right]\right) \leq 4 \frac{\|c\|_{\infty}}{1-\beta}(\bar{\alpha} \beta)^n
$$
following the stability criteria from  \cite[Theorem 4.1]{mcdonald2020exponential}.
We establish the desired result.
\end{proof}

\subsection{Approximations and Learning}

Throughout the rest of this section, we assume that 
Assumption \ref{main_assumption} holds. Earlier, we have already shown that if Assumption 
\ref{main_assumption} is valid for the POMDP, 
then Assumption \ref{lect} is applicable 
for the belief-MDP. The following theorems 
assist in finding near-optimal policies 
for the average cost criterion; for further analysis see \cite[Theorem 5]{creggZeroDelayNoiseless} or \cite[Theorem 7.3.6]{yuksel2020control}. In particular, that a solution to the average cost optimality equation exists and this is arrived at via the vanishing discount method in our analysis earlier is critical for the applicability of the following result:

\begin{theorem}
a) Under Assumption \ref{lect}, $\lim _{\beta \uparrow 1}(1-\beta) J^*_\beta\left(x_0\right) \rightarrow \rho^*$, where $\rho^*$ is the optimal average cost. Furthermore, if $\gamma_\beta$ solves the discounted cost optimality equation:
$$J^*_\beta(x)=\min _{u \in \mathbb{U}}\left\{c(x, u)+\beta \int_{\mathbb{X}} J^*_\beta(y) \mathcal{T}(d y \mid x, u)\right\}.$$
Then, for every $\epsilon>0$, there exists $\beta_{\epsilon} \in (0,1)$ such that for $\beta \in [\beta_{\epsilon},1)$,
$$
J\left(x, \gamma_\beta\right)-\rho^*<\epsilon.
$$
This implies that the discounted cost optimal 
policy is near-optimal for the average cost criterion.

b) Under Assumption \ref{lect}, let $\beta_\epsilon$ be chosen as above, such that with $h_{\beta}$ as given in (\ref{relativeValue}),
$$
\left|\rho^*-\left(1-\beta_\epsilon\right) J_{\beta_\epsilon}\left(x_0\right)\right| \leq \frac{\epsilon}{2}, \quad \mathrm{and} \quad \left(1-\beta_\epsilon\right)\left\|h_{\beta_\epsilon}\right\|_{\infty} \leq \frac{\epsilon}{2},
$$
so that $\gamma_{\beta_\epsilon}$ is $\epsilon$-optimal. 
Suppose that $\gamma_{\beta_\epsilon}^\delta$ is 
$\delta$-optimal policy for $\beta_\epsilon$ discounted
cost criteria, satisfying
$$
J_{\beta_\epsilon}\left(x, \gamma_{\beta_\epsilon}^\delta\right)-J_{\beta_\epsilon}(x)<\delta .
$$
Then,
$$
J\left(x, \gamma_{\beta_\epsilon}^\delta\right)-\rho^*<\epsilon+\delta
$$
That is a near-optimal discounted cost policy is also near-optimal for the average cost criterion.
\end{theorem}

Since Assumption \ref{main_assumption} for the POMDP 
implies Assumption \ref{lect} for the belief-MDP, 
as proven in the proof of Theorem \ref{main}, 
the above theorem also holds under Assumption \ref{main_assumption}.
Therefore, under Assumption \ref{main_assumption},
a near-optimal policy for the discounted cost criteria
is also near-optimal for the average 
cost criteria. The subsequent subsections will 
detail how to establish a near-optimal policy 
for discounted costs.

We have that $\U$ is a compact space. As demonstrated in \cite[Theorem 3.2]{SaLiYuSpringer}, if the transition kernel is weakly continuous
and the cost function is bounded and continuous,  then the optimal policy for a quantized action model 
in a discounted context is also near-optimal for 
the original model. This allows us to consider 
$\U$ as finite to identify a near-optimal policy.

\subsubsection{Near Optimality of Finite Models 
via Quantization for Average Cost}\label{Near_opt_quant}

In this subsection, we focus on achieving a 
near-optimal policy by quantizing the states 
of the belief MDP, namely $P(\X)$.
We follow the approach and results from 
\cite{KSYContQLearning}.

To create a new finite MDP model, we begin by 
quantizing the belief states. We select disjoint 
sets $\left\{Z_i\right\}_{i=1}^M$ such that 
$\bigcup_i Z_i=\Z$, and each 
$Z_i$ is distinct from $Z_j$ 
for any $i \neq j$.
For each set, we choose a representative state,
denoted as $z_i \in Z_i$.
This results in a finite state space for our 
model, represented by
$\bar{Z}:=\left\{z_1, \ldots, z_M\right\}$.
The quantization function maps the original state space
$\Z$ to this finite set $\bar{Z}$ as follows:
$$
q(z)=z_i \quad \text { if } z \in Z_i .
$$
To define the cost function, we select a 
weighting measure $\pi^* \in P(\Z)$ 
over $\Z$ such that 
$\pi^*\left(Z_i\right)>0$ for all 
$Z_i$.
Under Assumption \ref{main_assumption},
we know that $\Z$ is compact under $W_1$ metric.
We then define normalized measures 
for each quantization bin $Z_i$
using the weighting measure as:
$$
\hat{\pi}_{z_i}^*(A):=\frac{\pi^*(A)}{\pi^*\left(Z_i\right)}, 
\quad \forall A \subset Z_i, \quad \forall i \in\{1, \ldots, M\}.
$$
This normalized measure, $\hat{\pi}_{z_i}^*$,
is specific to the set 
$Z_i$ containing $z_i$.

Next, we define the stage-wise cost and the 
transition kernel for the MDP with the 
finite state space $\bar{Z}$ 
using these normalized weight measures.
For any $z_i, z_j \in \bar{Z}$ and 
$u \in \mathbb{U}$, the stage-wise cost function
and the transition kernel are:
$$
\begin{aligned}
c^*\left(z_i, u\right) & 
=\int_{Z_i} \tilde{c}(z, u) \hat{\pi}_{z_i}^*(d z), \\
\eta^*\left(z_j \mid z_i, u\right) & 
=\int_{Z_i} \eta\left(Z_j \mid z, u\right) 
\hat{\pi}_{z_i}^*(d z) .
\end{aligned}
$$

After establishing the finite state space 
$\bar{Z}$, the cost function $c^*$ and the transition kernel 
$\eta^*$, we introduce the discounted optimal 
value function for this finite model, denoted as 
$\hat{J}_\beta: \bar{Z} \rightarrow \mathbb{R}$.
We extend this function to the entire original state space 
$\Z$  by keeping it constant within the 
quantization bins. Therefore, for any 
$z \in Z_i$, we define:
$$
\hat{J}_\beta(z):=\hat{J}_\beta(z_i) .
$$

We also define the maximum loss function among 
the quantization bins as:
\begin{align}\label{L_max}
\bar{L}:=\max _{i=1, \ldots, M} \sup _{z, z^{\prime} \in Z_i}W_1(z,z^{\prime}).
\end{align}
\begin{assumption}[\cite{KSYContQLearning} Assumption 4]
\label{quantized_as} \noindent
    \begin{enumerate}
        \item $\Z$ is compact.
        \item There exists $\alpha_c>0$ such that 
        $\left|\tilde{c}(z, u)-\tilde{c}\left(z^{\prime},
         u\right)\right| \leq \alpha_c d(z,z^{\prime})$ 
         for all $z, z^{\prime} \in \Z$ and for all $u \in \mathbb{U}$.
        \item There exists $\alpha_\eta>0$ such that 
        $W_1\left(\eta(\cdot \mid z, u), 
        \eta\left(\cdot \mid z^{\prime}, 
        u\right)\right) \leq \alpha_\eta d(z,z^{\prime})$ 
        for all $z, z^{\prime} \in \Z$ 
        and for all $u \in \mathbb{U}$.
    \end{enumerate}

\end{assumption}

The following theorem states that an 
optimal policy of the 
quantized model is near-optimal for the 
original 
model as $\bar{L} \rightarrow 0$.

\begin{theorem}[\cite{KSYContQLearning} Theorem 6]
Under Assumption \ref{quantized_as}, we have
$$
\sup _{z \in \Z}\left|J_\beta\left(z, 
\hat{\gamma}\right)-J_\beta^*\left(z\right)\right| 
\leq \frac{2 \alpha_c}{(1-\beta)^2\left(1-\beta 
\alpha_\eta\right)} \bar{L} ,
$$
where $\bar{L}$ is defined in (\ref{L_max}) and 
$\hat{\gamma}$ denotes the optimal policy of 
the finite-state approximate model extended to 
the state space $\Z$ via the quantization function $q$.
\end{theorem}
A similar result is presented in the 
\cite{SaLiYuSpringer}, Theorem 4.38, offering
a slightly weaker bound.

Under Assumption \ref{main_assumption}, belief MDP satisfies
 Assumption \ref{quantized_as} because \( \Z \) is compact 
 under \( W_1 \) metric. Due to inequality (\ref{c_tilde_cont}), 
 we have \( |\tilde{c}(z, u)-\tilde{c}(z', u)| \leq K_1 W_1(z,z') \)
for all \( z, z' \in \Z \) and for all \( u \in \mathbb{U} \). 
Theorem \ref{ergodicity} implies \( W_1(\eta(\cdot \mid z, u), 
\eta(\cdot \mid z', u)) \leq K_2 W_1(z,z') \) for all 
\( z, z' \in \Z \) and for all \( u \in \mathbb{U} \). 
Thus, for belief MDP, quantization provides the following bound:
\[ \sup _{z \in \Z}\left|J_\beta\left(z, \hat{\gamma}\right)-J_\beta^*\left(z\right)\right| \leq \frac{2 K_1}{(1-\beta)^2(1-\beta K_2)} \bar{L} . \]
Furthermore, quantized model gives near-optimal policy of the 
original belief MDP model as $\bar{L} \rightarrow 0$. 

{\bf Q-learning.} This section introduces the Q-iteration method to 
identify the optimal policy for a quantized 
belief Markov Decision Process (MDP). We apply 
this method to the quantized belief MDP with
finite state space, $\bar{Z}$, and a finite action 
space, $\U$. The Q-learning process updates Q-functions 
as follows:
For each time step $t \geq 0$, if the current 
state-action pair is $\left(Z_t, U_t\right) = (z, u)$, 
the Q-value for this pair is updated in the following manner:
\begin{align}\label{Q_alg_quant}
&Q_{t+1}\left(Z_t, U_t\right)=
(1- \alpha_t\left(Z_t, U_t\right)) 
Q_t\left(Z_t, U_t\right) 
\\\nonumber &+\alpha_t\left(Z_t, U_t\right)\left(c^*\left(Z_t, U_t\right)+\beta \min _{v \in \mathbb{U}} Q_t\left(Z_{t+1}, v\right)\right)
\end{align}
Key assumptions for this Q-learning approach include:
\begin{assumption}[\cite{KSYContQLearning} Assumption 5]
\label{Q_learning_quant}\noindent
\begin{itemize}
\item 
The learning rate $\alpha_t(z, u)=0$ if 
$\left(Z_t, U_t\right)\neq(z, u)$, otherwise it is defined as:
$$
\alpha_t(z, u)=\frac{1}{1+\sum_{k=0}^t \1_{\left\{Z_k=z, U_k=u\right\}}} .
$$
\item Under the exploration policy $\gamma^*, Z_t$ is uniquely ergodic and thus has a unique invariant invariant measure $\pi_{\gamma^*}$.
\item  During exploration, every possible 
state-action pair in $\bar{Z} \times \U$ is visited 
an infinite number of times\footnote{This can be relaxed and refined to concern those states in the support of an invariant measure under an exploration policy \cite{karayukselNonMarkovian}.}.
\end{itemize}

Under Assumption \ref{Q_learning_quant}, 
the Q-learning algorithm (\ref{Q_alg_quant} )converges to the 
fixed-point solution $Q^*$. A stationary policy, 
$\gamma^N$, that selects actions to minimize the Q-value 
at each state, i.e., $\gamma^N(z) \in \min_{u} Q^*(z, u)$, 
is optimal. This method enables determining the optimal 
policy for the quantized belief MDP \cite{KSYContQLearning}.

\end{assumption}

\subsubsection{Near Optimality of Finite Window Policies
for Average Cost}\label{Near_opt_fW}
In this subsection, we focus on the
finite window history to obtain a 
near-optimal policy for the 
average cost criteria. Throughout this 
subsection, we assume $\Y,\U$ to be finite. 
Assuming $\Y,\U$ as finite, we use 
finite window information to obtain finite state MDP and demonstrate 
that the near-optimal policy of this finite MDP is 
near-optimal for POMDP. We follow approach and results
from \cite{kara2021convergence}.

Recall $z_n$ is the 
belief distribution defined as
$z_n(\cdot)=P^{\mu}\{X_n \in \cdot \mid y_0, 
\ldots, y_n, u_0,$ $\ldots, u_{n-1}\} \in \mathcal{P}(\mathbb{X})$,
where the initial state $X_0$ has a prior distribution $\mu \in \mathcal{P}(\mathbb{X})$.

\begin{definition}
    $z^-_{n}$ is the posterior distribution at time $n$ before 
    observing $Y_n$ and is defined as 
$z_n^-(\cdot):=P^\mu\left\{X_n \in \cdot \mid y_0, 
\ldots, y_{n-1}, u_0, \ldots, u_{n-1}\right\} \in \mathcal{P}(\mathbb{X})$,
where the initial state $X_0$ has a prior distribution $\mu \in \mathcal{P}(\mathbb{X})$.
\end{definition}

For any $N,n \geq 0$, we can determine $z_{n+N}$ as follows: 
$$z_{n+N}=P\left\{X_{n+N} \in \cdot \mid z_n^-, y_n, 
\ldots, y_{n+N}, u_n, \ldots, u_{n+N-1}\right\}. $$

We then consider an alternative finite window 
belief MDP reduction. Let us define the 
state variable at time $n \geq N$ as:
\begin{align}\label{FWBState}
\hat{z}_n=\left(z_{n-N}^{-}, I_n^N\right),
\end{align}
where, for $N \geq 1$, the components are:
\begin{align}\label{newStateF}
z_{n-N}^{-} & =\operatorname{Pr}\left(X_{n-N} \in \cdot \mid y_{n-N-1}, \ldots, y_0, u_{n-N-1}, \ldots, u_0\right), \\
I_n^N & =\left\{y_n, \ldots, y_{n-N}, u_{n-1}, \ldots, u_{n-N}\right\},
\end{align}
and for $N=0$, $I_n^N=y_n$ with the prior measure $\mu$ on $X_0$.
The state space is thus 
$\hat{\mathcal{Z}}=\Z 
\times \mathbb{Y}^{N+1} \times \mathbb{U}^N$.

The natural mapping between state spaces 
is defined by $\psi: \hat{\mathcal{Z}} \rightarrow \Z$, 
such that:
\begin{align}
\psi\left(\hat{z}_n\right)=
\psi\left(z_{n-N}^{-}, I_n^N\right) & =
P^{z_{n-N}^{-}}\left(X_n \in \cdot \mid y_{n-N}, 
\ldots, y_n, u_{n-N-1}, \ldots, u_{n-N-1}\right)=z_n
\end{align}

The new transition kernel and cost function are defined as:
\begin{align}
\hat{\eta}(\cdot \mid \hat{z}, u)=
\int_{\mathbb{Y}} \1_{\{(z_{n-N+1}^{-}, I_{n+1}^N) \in \cdot\}}
 \hat{P}(d y \mid \hat{z}, u),
\end{align}
where 
$$\hat{P}(\cdot \mid \hat{z}, u)=
\operatorname{Pr}\left\{Y_{n+1} \in \cdot 
\mid Z_{n-N}=z_{n-N}^{-}, I_n^N, U_n=u\right\}$$ 
from $\mathcal{Z} \times \mathbb{U}$ to $\mathbb{Y}$.
The cost function is:
\begin{align*} 
    & \hat{c}\left(\hat{z}_n, u_n\right)
    =\tilde{c}\left(\psi\left(z_{n-N}^{-}, I_n^N
    \right), u_n\right)=
    \\ & \int_{\mathbb{X}} c
    \left(x_n, u_n\right) P^{z_{n-N}^{-}}\left(
    d x_n \mid y_{n-N}, \ldots, y_n, u_{n-1},
    \ldots, u_{n-N}\right).
\end{align*}

If $\gamma$ is an optimal policy 
of belief MDP then $\psi^{-1}(\gamma)$
is an optimal policy of finite window belief MDP
\cite{kara2021convergence}.

Next, we explain how to derive a near-optimal 
policy for the finite window belief MDP in 
the context of discounted cost. We introduce a 
new approximate MDP for this purpose.

For $n \geq N$, fixing $z_{n-N}^-$ to a constant  
$z\in P(\X)$, we obtain a 
new MDP with state $\hat{z}_n=\left(z, I_n^N\right)$
with state space 
$\hat{\Z}^N:=\{z\} \times \mathbb{Y}^{N+1} \times \mathbb{U}^N$,
cost function
$
\hat{c}^N\left(\hat{z}_n^N, u_n\right) 
:=\hat{c}\left((z, I_n^N), u_n\right)
$ and
transition kernel
$
\hat{\eta}^N\left(\hat{z}_{n+1}^N \mid 
\hat{z}_n^N, u_n\right):=
\hat{\eta}\left(\Z\times I_{n+1}^N \mid (z, I_n^N), u_n\right)
$ \cite{kara2021convergence}.

This process, being finite state and fully observable, 
allows for finding an optimal policy $\phi^N$ through 
$Q$-iteration. This policy can be extended to 
$\hat{\Z}$ as $\tilde{\phi}^N(\hat{z}^N_n)=
\tilde{\phi}^N(z_{n-N}^-, I_n^N):=\phi^N(z, I_n^N)$.

The following theorem indicates that this policy 
is also nearly optimal for
 $\hat{z}_n$:
\begin{theorem}[\cite{kara2021convergence} Theorem 3]
For $\hat{z}_0=P^{z_0^{-}}\left(X_n\in . \mid I_0^N\right)$, 
with a policy $\hat{\gamma}$ acting on the first 
$N$ steps which produces $I_0^N={Y{[0,N]},U_{[0,N-1]}}$, 
the following holds:
$$
E_{z_0^{-}}^{\hat{\gamma}}\left[\left|\tilde{J}_\beta^N\left(\hat{z}_0, \tilde{\phi}^N\right)
-J_\beta^*\left(\hat{z}_0\right)\right| \mid I_0^N\right] 
\leq \frac{2 \|c\|_{\infty}}{(1-\beta)} \sum_{t=0}^{\infty} \beta^t L^N_t,
$$
where
\begin{align*}
&L_t^N:=\sup _{\hat{\gamma} \in \hat{\Gamma}} E_{z_0^{-}}^{\hat{\gamma}}\left[ \right. \left\| \right. P^{z_t^{-}}\left(X_{t+N} \in \cdot \mid Y_{[t, t+N]}, U_{[t, t+N-1]}\right)-
\\ & \qquad \qquad \qquad \qquad \qquad P^{z}\left(X_{t+N} \in \cdot \mid Y_{[t, t+N]}, U_{[t, t+N-1]}\right)  \left.\right\|_{T V} \left.\right]
\end{align*}
\end{theorem}

Here, $\hat{\gamma}\in\hat{\Gamma}$, where $\hat{\Gamma}$ can be taken to be those Markov control policies under controlled states given with $\hat{z}_n$ defined in (\ref{FWBState}); that is, policies which map $(\hat{z}_n,n) \to u_n$ for all $n \in \mathbb{Z}_+$. 

With filter stability \cite{kara2021convergence}, as $N$ approaches infinity, 
the upper bound tends to zero. Given filter stability 
and a sufficiently large $N$, the policy $\tilde{\phi}^N$ 
becomes near-optimal for the finite window belief 
MDP for discounted cost criteria.
If $\beta$ is sufficiently large, this policy also 
becomes near-optimal for the average cost criterion.

%

{\bf Q-learning.} We now outline how to find all optimal policies 
using Q-learning for the approximate finite belief 
MDP in the context of discounted cost. As the 
posterior distribution $z$ is fixed, we track 
observations and actions. We assume tracking of 
the last $N+1$ observations and the last $N$ 
control actions after at least $N+1$ time steps. 
Thus, at time $n$, we monitor the information 
variables $I_n^N$.

The Q-value iteration is constructed using 
these information variables. For these new 
approximate states, we follow the standard 
Q-learning algorithm. For any 
$I \in \mathbb{Y}^{N+1} \times \mathbb{U}^N$ and 
$u \in \mathbb{U}$, the Q-value update is as follows:
\begin{align}\label{Q_l_alg_finite}
Q_{t+1}(I, u)=\left(1-\alpha_t(I, u)\right) 
Q_t(I, u)+\alpha_t(I, u)\left(\hat{c}^N(I, u)+
\beta \min _v Q_t\left(I_1^t, v\right)\right),
\end{align}
where $I_1^t=\left\{Y_{t+1}, y_t, \ldots, y_{t-N+1},
 u_t, \ldots, u_{t-N+1}\right\}$.
Exploration policies are employed, randomly 
choosing control actions independently, 
such that at time $t$, the action $u_t$ is 
selected with probability $\sigma_i$ for 
each $u_i \in \mathbb{U}$, where $\sigma_i > 0$ for all $i$.
The following assumption is a revision of 
Assumption \ref{Q_learning_quant}, specifically 
adapted for the finite window context:
\begin{assumption}
    [\cite{kara2021convergence} Assumption 4.1]
    \label{Q_assumption_finite}
    \noindent
\begin{enumerate}
\item $\alpha_t(I, u)=0$ unless 
$\left(I_t, u_t\right)=(I, u)$. In other cases,
$$
\alpha_t(I, u)=\frac{1}{1+\sum_{k=0}^t \1_{\left\{I_k=I, u_k=u\right\}}}
$$
\item Under the stationary (memoryless or finite memory exploration)
policy, say $\gamma$, the true state process, 
$\left\{X_t\right\}_t$, is positive Harris 
recurrent and in particular admits a unique 
invariant measure $\pi_\gamma^*$.
\item Furthermore, we have that $P(Y_t = y | x) > 0$ for every $x \in \mathbb{X}$, and thus during the exploration phase, every 
$(I, u)$ pair is visited infinitely often.
\end{enumerate}
\end{assumption}
\begin{theorem}
    [\cite{kara2021convergence}  
    Theorem 4.1, Corollary 5.1]
Suppose the following conditions hold:
\begin{enumerate}
\item Assumption \ref{Q_assumption_finite} holds.
\item The POMDP is such that the filter is 
stable uniformly over priors in expectation 
under total variation, meaning
$L_t \rightarrow 0$ as $N \rightarrow \infty$.
\end{enumerate}    
Then, the followings are true:
\begin{itemize}
\item The algorithm given in (\ref{Q_l_alg_finite}) 
converges almost surely to $Q^*$ which satisfies 
fixed point equation.
\item A stationary policy $\gamma^N$ that
 satisfies 
$\gamma^N(I) \in \min _u Q^*(I, u)$ 
is an optimal policy.
\end{itemize} 
\end{theorem}
\begin{note}\label{Curtis_exp}
    \cite{mcdonald2020exponential}, Theorem 4.1 provides 
    sufficient conditions for uniform filter stability. 
    If we have 
    $\bar{\alpha}=(1-\tilde{\delta}(\T))(2-\delta(Q)) < 1$, 
    then the filter is exponentially stable with 
    coefficient $\bar{\alpha}$ for any control policy. 
    Here, $\tilde{\delta}(\T) = \inf_{u \in \mathcal{U}} 
    \delta(T(\cdot | \cdot, u))$. See also an analysis via the Hilbert projective metric \cite{le2004stability,mcdonald2020exponential}. 
\end{note}
This method allows for obtaining 
near-optimal policies of POMDP for the discounted cost, 
and consequently under Assumption \ref{main_assumption} 
for average cost. It is important to note that for 
the average cost optimal policy, the actions taken 
in the first $N$ steps are not significant. 
The cost incurred during these initial steps does 
not impact the overall outcome, and since the optimal 
cost for each state is constant, applying 
the optimal policy after $N$ steps will still 
yield an optimal policy.

\section{Concluding Remarks}

The average cost optimality of partially observable Markov decision process
is a challenging problem. In this paper, we presented explicit and easily testable conditions for the existence of solutions to 
the average cost optimality equation where the state space is compact. A comparison with the related literature and several examples are presented. Notably, our paper appears to be the first one where a contraction result is presented and utilized for the optimality analysis. We finally presented several implications of our analysis and existence result on approximations, learning, and robustness.

\appendix
\section{Proof of Theorem \ref{ergodicity}} 
\begin{proof}
  We equip $\mathcal{Z}$ with the metric $W_1$ to define the Lipschitz seminorm $\norm{f}_L$ of any Borel measurable function $f:\mathcal{Z}\to \mathbb{R}$.
      \begin{align}
      &W_{1}\left(\eta(\cdot \mid z_0,u), \eta\left(\cdot \mid z_0^{\prime},u\right)\right)\nonumber\\
      & =\sup _{f \in \operatorname{Lip}(\mathcal{Z},1), \;\norm{f}_\infty\leq D/2}\bigg|\int_{\mathbb{Y}} f\left(z_1\left(z_0^{\prime},u, y_1\right)\right) P\left(d y_1 \mid z_0^{\prime},u\right)\\
      &-\int_{\mathbb{Y}} f\left(z_1\left(z_0, u, y_1\right)\right) P\left(d y_1 \mid z_0, u\right)\bigg|\label{kappabeta1}
     \end{align}
  
     For any $f:\mathcal{Z}\to \mathbb{R}$ such that $\norm{f}_L\leq 1$ and $\norm{f}_\infty\leq D/2$, we have 
      \begin{align}
      &\left|\int_{\mathbb{Y}} f\left(z_1\left(z_0^{\prime}, u, y_1\right)\right) P\left(d y_1 \mid z_0^{\prime}, u\right)-\int_{\mathbb{Y}} f\left(z_1\left(z_0, u,  y_1\right)\right) P\left(d y_1 \mid z_0, u\right)\right| \nonumber\\
      &\leq \left|\int_{\mathbb{Y}} f\left(z_1\left(z_0^{\prime}, u, y_1\right)\right) P\left(d y_1 \mid z_0^{\prime}, u\right)-\int_{\mathbb{Y}} f\left(z_1\left(z_0^{\prime}, u, y_1\right)\right) P\left(d y_1 \mid z_0, u\right)\right| \nonumber\\
      &+ \int_{\mathbb{Y}}\left|f\left(z_1\left(z_0^{\prime}, u, y_1\right)\right)-f\left(z_1\left(z_0, u, y_1\right)\right)\right| P\left(d y_1 \mid z_0, u\right) \nonumber\\
      &\leq\frac{D}{2}\left\|P\left(\cdot \mid z_0^{\prime}, u \right)-P\left(\cdot \mid z_0, u\right)\right\|_{T V}\\
      &+ \int_{\mathbb{Y}}\left|f\left(z_1\left(z_0^{\prime}, u, y_1\right)\right)-f\left(z_1\left(z_0, u, y_1\right)\right)\right| P\left(d y_1 \mid z_0, u\right).\label{second}
      \end{align}
      
      For the first term:
      \begin{align}\label{TV1}
      &\left\|P\left(\cdot \mid z_0^{\prime}, u\right)-P\left(\cdot \mid z_0, u\right)\right\|_{T V}\\
      &=\sup _{\|g\|_{\infty} \leq 1}\left|\int g\left(y_1\right) P\left(d y_1 \mid z_0^{\prime}, u\right)-\int g\left(y_1\right) P\left(d y_1 \mid z_0, u\right)\right|\nonumber \\
      &\leq (1-\delta(Q))\left\|\mathcal{T}\left(d x_1 \mid z_0^{\prime}, u\right)-\mathcal{T}\left(d x_1 \mid z_0, u\right)\right\|_{T V}
      \end{align}
      by the Dobrushin Contraction Theorem \cite{dobrushin1956central}.

      \begin{align}\label{TV2}
      &\left\|\mathcal{T}\left(d x_1 \mid z_0^{\prime}, u\right)-\mathcal{T}\left(d x_1 \mid z_0, u\right)\right\|_{T V}
      \\ & = \sup _{\|g\|_{\infty} \leq 1}\left(\int g\left(x_1\right) T\left(d x_1 \mid z_0^{\prime}, u\right)-\int g\left(x_1\right) T\left(d x_1 \mid z_0, u\right)\right)\nonumber\\
        &= \sup _{\|g\|_{\infty} \leq 1}\left(\int\Tilde{g_g}(x_0)z_0^{\prime}(dx_0)-\Tilde{g_g}(x_0)z_0(dx_0)\right)
      \end{align}
      where
      \begin{align*}
          \Tilde{g_g}(x)= \int g\left(x_1\right) T\left(d x_1 \mid x, u\right).
      \end{align*}
      
      For all $x_0^{\prime}, x_0\in \mathbb{X}$, 
      we have, 
      \begin{align*}
      \left\|\mathcal{T}\left(d x_1 \mid x_0^{\prime}, u\right)-\mathcal{T}\left(d x_1 \mid x_0, u\right)\right\|_{T V}\leq \alpha d(x_0,x_0^{\prime}).
      \end{align*}
      As a result, we get $\Tilde{g}/\alpha \in Lip(\mathbb{X},1)$.
  
      Then, by inequality (\ref{TV1}) and (\ref{TV2}) we can write 
      \begin{align}\label{TV}
      \left\|P\left(\cdot \mid z_0^{\prime}, u\right)-P\left(\cdot \mid z_0, u\right)\right\|_{T V} \leq \alpha(1-\delta(Q))W_{1}\left(z_0^{\prime}, z_0\right).
      \end{align}
      Finally, we can analyze the second term in (\ref{second})
      \begin{align}\label{eq3}
      &\int_{\mathbb{Y}}\left|f\left(z_1\left(z_0^{\prime}, u, y_1\right)\right)-f\left(z_1\left(z_0, u, y_1\right)\right)\right| P\left(d y_1 \mid z_0, u\right)\nonumber\\
      &\leq \int_{\mathbb{Y}}W_1(z_1\left(z_0^{\prime}, u, y_1\right), z_1\left(z_0, u, y_1\right)) P\left(d y_1 \mid z_0, u\right)\nonumber\\
      &=\int_{\mathbb{Y}} \sup_{g \in \operatorname{Lip}(\mathbb{X})}\left(\int_{\mathbb{X}} g(x_1)w_{y_1}(dx_1)\right)P\left(d y_1 \mid z_0, u\right),
     \end{align}
      where $w_{y_1}=\left(z_1\left(z_0^\prime, u, y_1\right)- z_1\left(z_0, u, y_1\right)\right)$ 
      which is a signed measure on $\mathbb{X}$. 
      By the Measurable Selection Theorem
      \footnote{[\cite{himmelberg1976optimal}, Theorem 2][Kuratowski Ryll-Nardzewski Measurable Selection Theorem]
Let $\mathbb{X}, \mathbb{Y}$ be Polish spaces and $\Gamma=(x, \psi(x))$ where $\psi(x) \subset \mathbb{Y}$ be such that, $\psi(x)$ is closed for each $x \in \mathbb{X}$ and let $\Gamma$ be a Borel measurable set in $\mathbb{X} \times \mathbb{Y}$. Then, there exists at least one measurable function $f: \mathbb{X} \rightarrow \mathbb{Y}$ such that $\{(x, f(x)), x \in \mathbb{X}\} \subset \Gamma$.
}
       choose measurable 
      $$g_y\in \arg\sup_{g \in \operatorname{Lip}(\mathbb{X},1)}\left(\int_{\mathbb{X}} g(x)w_{y}(dx)\right).$$ 
  
      After that we can continue with the equation (\ref{eq3}),
      \begin{align}
      &\int_{\mathbb{Y}} \sup_{g \in \operatorname{Lip}(\mathbb{X},1)}\left(\int_{\mathbb{X}} g(x_1)w_{y_1}(dx_1)\right)P\left(d y_1 \mid z_0, u\right)\nonumber\\
      &=\int_{\mathbb{Y}} \int_{\mathbb{X}} g_{y_1}(x_1)z_1\left(z_0^\prime, u, y_1\right)(dx_1) P(d y_1 \mid z_0, u)\\
      &-\int_{\mathbb{Y}}\int_{\mathbb{X}} g_{y_1}(x_1)z_1\left(z_0^\prime, u, y_1\right)(dx_1) P(d y_1 \mid z_0^\prime, u)\label{gy1}\\
      &+\int_{\mathbb{Y}} \int_{\mathbb{X}} g_{y_1}(x_1)z_1\left(z_0^\prime, u, y_1\right)(dx_1) P(d y_1 \mid z_0^\prime, u)\\
      &-\int_{\mathbb{Y}}\int_{\mathbb{X}} g_{y_1}(x_1)z_1\left(z_0, u, y_1\right)(dx_1) P(d y_1 \mid z_0, u)\label{gy2}
      \end{align}
      
      For the first term, we can write by the same argument as earlier
      \begin{align*}
      \norm{\int_{\mathbb{X}} g_{y_1}(x_1)z_1\left(z_0^\prime, u, y_1\right)(dx_1)}_\infty\leq\norm{g_{y_1}}_\infty \leq D/2 .
      \end{align*}
      So,
      \begin{align}\label{sup-1}
      &\int_{\mathbb{Y}} \int_{\mathbb{X}} g_{y_1}(x_1)z_1\left(z_0^\prime, u, y_1\right)(dx_1) P(d y_1 \mid z_0, u)\\
      &-\int_{\mathbb{Y}}\int_{\mathbb{X}} g_{y_1}(x_1)z_1\left(z_0^\prime, u, y_1\right)(dx_1) P(d y_1 \mid z_0^\prime, u)\nonumber\\
      &\leq \frac{D}{2} \left\|P\left(\cdot \mid z_0^{\prime}, u\right)-P\left(\cdot \mid z_0, u\right)\right\|_{T V}\nonumber\\
      & \leq \alpha \frac{D}{2} (1-\delta(Q))W_{1}\left(z_0^{\prime}, z_0\right)
      \end{align}
      by inequality (\ref{TV}).
      
      For the second term,  we can write by smoothing
      \begin{align}\label{sup-2}
      &\int_{\mathbb{Y}} \int_{\mathbb{X}} g_{y_1}(x_1)z_1\left(z_0^\prime, u, y_1\right)(dx_1) P(d y_1 \mid z_0^\prime, u)\\
      &-\int_{\mathbb{Y}}\int_{\mathbb{X}} g_{y_1}(x_1)z_1\left(z_0, u, y_1\right)(dx_1) P(d y_1 \mid z_0, u)\nonumber\\
     &=\int_{\mathbb{X}} \omega(x_1) \mathcal{T}\left(d x_1 \mid z_0^{\prime}, u\right)-\int_{\mathbb{X}}\omega(x_1) \mathcal{T}\left(d x_1 \mid z_0, u\right)
      \end{align}
      where $$\omega(x_1)=\int_{\mathbb{Y}} g_{y_1}(x_1)Q\left(d y_1 \mid x_1\right).$$
     
      For any $x^\prime, x^{\prime\prime} \in \mathbb{X}$, 
      \begin{align*}
      &\int_{\mathbb{X}} \omega(x) \mathcal{T}\left(d x \mid x^{\prime\prime}, u\right)-\int_{\mathbb{X}}\omega(x) \mathcal{T}\left(d x \mid x^{\prime}, u\right)\\
      &\leq \norm{\omega}_\infty \left\|\mathcal{T}(\cdot \mid x^{\prime\prime}, u)-\mathcal{T}\left(\cdot \mid x^{\prime}, u\right)\right\|_{T V}\\
      &\leq \norm{\omega}_\infty\alpha d(x^{\prime\prime},x^{\prime})\\
      &\leq\alpha \frac{D}{2} d(x^{\prime\prime}, x^{\prime})
      \end{align*}
      So, by definition of the $W_1$ norm (\ref{defkappanorm}), we have
      \begin{align}\label{sup-2-1}
      &\int_{\mathbb{X}} \omega(x_1) \mathcal{T}\left(d x_1 \mid z_0^{\prime}, u\right) \nonumber \\
  & \qquad    -\int_{\mathbb{X}}\omega(x_1) \mathcal{T}\left(d x_1 \mid z_0, u\right)\nonumber\\
      &=\int_{\mathbb{X}}\int_{\mathbb{X}} \omega(x_1) \mathcal{T}\left(d x_1 \mid x_0, u\right)z^\prime_0(dx_0) \nonumber \\
  & \qquad -\int_{\mathbb{X}} \int_{\mathbb{X}}\omega(x_1) \mathcal{T}\left(d x_1 \mid x_0, u\right)z_0(dx_0)\nonumber\\
      &\leq  \alpha \frac{D}{2} W_{1}\left(z_0, z_0^{\prime}\right).
      \end{align}
      So, by the inequalities (\ref{gy1}), (\ref{gy2}), (\ref{sup-1}), (\ref{sup-2}), (\ref{sup-2-1}) we get
      \begin{align}\label{last2}
      & \int_{\mathbb{Y}}\left|f\left(z_1\left(z_0^{\prime}, u, y_1\right)\right)-f\left(z_1\left(z_0, u, y_1\right)\right)\right| P\left(d y_1 \mid z_0, u\right) \nonumber\\
      &\leq \alpha \frac{D}{2} (2-\delta(Q))W_{1}\left(z_0^{\prime}, z_0\right).
      \end{align}
      If we take the supremum of the equation 
     over all $f\in Lip(\mathcal{Z})$, 
     then by using the inequalities 
     (\ref{second}),(\ref{TV}), and (\ref{last2}), 
     we can write:
      \begin{align}\label{imp}
      &W_{1}\left(\eta(\cdot \mid z_0, u), \eta\left(\cdot \mid z_0^{\prime}, u\right)\right)
      \leq \left(\frac{\alpha D (3-2\delta(Q))}{2}\right)W_{1}\left(z_0, z_0^{\prime}\right)\end{align}
      \end{proof}

\section{Proof of Lemma \ref{exisACOE}}\label{Acoe_proof}
We present a specialization of \cite[Theorem 7.3.3]{yuksel2020control} to the compact case as in the paper we have that $\Z$ is compact.

Let us first recall the Arzelà-Ascoli theorem.
\begin{theorem}[\cite{Dud02} Theorem 2.4.7]\label{Arzoli-Ascali} Let $F$ be an 
equi-continuous family of functions 
on a compact space $\mathbb{X}$ and let 
$h_n$ be a sequence in $F$ such that the 
range of $f_n$ is compact. Then, there 
exists a subsequence $h_{n_k}$ which 
converges uniformly to a continuous 
function. 
\end{theorem}
\begin{proof}[Proof of Lemma \ref{exisACOE}]
First, since the cost function $\tilde{c}$
is bounded
(assume it is bounded by $M<\infty$)
the expression 
$(1-\beta)J_\beta(z)$ 
is uniformly bounded by $M$
for every $\beta \in (0,1)$ and $z\in \Z$.
By the Bolzano–Weierstrass theorem, for a fixed
 $z$ and any sequence 
$\beta \uparrow 1$,
there exists a subsequence $\beta(k)\uparrow 1$ such that
$\left(1-\beta(k)\right) J_{\beta(k)}\left(z\right) 
\rightarrow \rho^*$ for some $\rho^*$. 
Observe that for any $z \in \Z$
$$
\left(1-\beta(k)\right) J_{\beta(k)}(z)=
\left(1-\beta(k)\right)\left(J_{\beta(k)}(z)-
J_{\beta(k)}\left(z_0\right)\right)+
\left(1-\beta(k)\right) J_{\beta(k)}\left(z_0\right),
$$
which, by the uniform boundedness of $h_{\beta(k)}(z)=J_{\beta(k)}(z)-J_{\beta(k)}\left(z_0\right)$, 
implies that the limit $\rho^*$ does not depend on $z$.

By Assumption \ref{lect}-\ref{equic},
$h_{\beta(k)}$ is equicontinuous.
By
Theorem \ref{Arzoli-Ascali}, 
there exists a further subsequence of 
$h_{\beta(k)}$,
$\left\{h_{\beta(k_l)}\right\}$, which converges (uniformly on compact sets) to a continuous and bounded function $h$. 
Since $\U$ is compact and cost function is continuous
and transition kernel is weakly continuous,
we have by the Bellman equation:
\begin{align}
\label{J_diff} & 
J_\beta(z)-J_\beta\left(z_0\right)\nonumber \\ &
=\min _{u \in \mathbb{U}}\left(c(z, u)+\beta \int 
\eta\left(d z^{\prime} \mid z, u\right)
\left(J_\beta\left(z^{\prime}\right)-J_\beta\left(z_0\right)\right)-(1-\beta) 
J_\beta\left(z_0\right)\right)
\end{align}

Take the limit in (\ref{J_diff}) along 
the subsequence $\beta(k_l)$, we get
\begin{align}
\nonumber h(z) & =\lim _l \min _{\mathbb{U}}\left[c(z, u)+\beta\left(k_l\right) \int_{\Z} 
h_{\beta(k_l)}(y) \eta(d y \mid z, u) - (1-\beta(k_l)) 
J_{\beta(k_l)}\left(z_0\right) \right] \\
& =\lim _l \min _{\mathbb{U}}\left[c(z, u)+\beta\left(k_l\right) \int_{\Z} 
h_{\beta(k_l)}(y) \eta(d y \mid z, u) \right] - \rho^* \nonumber 
\end{align}

From this we obtain
\begin{align}\label{lim_h}
    \lim _l \min _{\mathbb{U}}\left[c(z, u)+\beta\left(k_l\right) \int_{\Z} 
    h_{\beta(k_l)}(y) \eta(d y \mid z, u) \right] 
    = h(z) + \rho^*
\end{align}

We now show that in the above, the order of limit and minimization can be swapped: 
Using the compactness of $\mathbb{U}$,
the continuity of $$c(z, u)+\beta(k_l) \int_{\Z} h_{\beta(k_l)}(y) \eta(d y | z, u)$$ on $\mathbb{U}$, 
and the equicontinuity of $\left\{h_{\beta(k)}\right\}$ we can define a sequence $u_l$ such that:
\begin{align*}
    & A_l:= c\left(z, u_l\right)+\beta\left(k_l\right) \int_{\Z} h_{\beta\left(k_l\right)}(y) \eta\left(d y \mid z, u_l\right)=\\
    &\min _{\mathbb{U}}\left[c(z, u)+\beta\left(k_l\right) \int_{\Z} h_{\beta\left(k_l\right)}(y) \eta(d y \mid z, u)\right] 
\end{align*}
By the compactness of the action space $\U$, 
there exists a further subsequence such that 
$u_{l_n} \rightarrow u^*$
for some $u^*$ along this 
further subsequence. 
By weak continuity of the kernel, 
we then have that 
$\eta\left(d y \mid z, u_{l_n}\right) \rightarrow \eta(d y \mid z, u^*)$. Since $h_{\beta\left(k_l\right)}$ is uniformly bounded, we have $\lim A_l = \lim A_{l_n}=c(z, u)+\int_{\Z} h(y) \eta(d y \mid z, u^*)$.


On the other hand, for any fixed action $\bar{u}$, we know that
$$c\left(z, \bar{u}\right)+\beta\left(k_l\right) \int_{\Z} h_{\beta\left(k_l\right)}(y) \eta\left(d y \mid z, \bar{u}\right)
\geq A_l$$
and taking the limit as $l \to \infty$, by the same argument, 
$$c(z, \bar{u})+\int_{\Z} h(y) \eta(d y \mid z, \bar{u})
\geq \lim A_l =c(z, u)+\int_{\Z} h(y) \eta(d y \mid z, u).$$
Therefore,
$$
\lim A_l = \min _{\mathbb{U}}\left[c(z, u)+\int_{\Z} h(y) \eta(d y \mid z, u)\right].$$
Combining this with equation (\ref{lim_h}), we obtain
$$h(z)+\rho^*=\min _{\mathbb{U}}\left[c(z, u)+\int_{\Z} h(y) \eta(d y \mid z, u)\right].$$
Thus, we have found a bounded solution to the ACOE equation, completing the proof.
\end{proof}

\bibliographystyle{unsrt}
\bibliography{SerdarBibliography.bib}

\end{document}